\documentclass[a4paper]{article}
\usepackage[margin=1in]{geometry}

\usepackage{mathtools}
\usepackage{amsmath,amssymb,amsfonts}%
\usepackage{amsthm}%
\usepackage{mathrsfs}%
\usepackage{textcomp}%
\usepackage{manyfoot}%
\usepackage{xparse}
\usepackage[normalem]{ulem}
\usepackage{subcaption}
\usepackage{graphicx}

\usepackage{hyperref}
\hypersetup{
	colorlinks,
	linkcolor=blue,
	citecolor=blue,
	urlcolor=blue,
}
\usepackage[capitalize]{cleveref}

\usepackage{import}
\usepackage{amsthm}
\usepackage{xparse}
\usepackage{import}
\usepackage[normalem]{ulem}

\makeatletter
\@ifpackageloaded{unicode-math}{%
	\newcommand{\mymathbold}{\symbf}%
}{%
	\usepackage{bm}%
	\newcommand{\mymathbold}{\bm}%
}
\makeatother

\newcommand{\scrbar}[1]{\overline{\mathcal{#1}}}
\newcommand{\scrhat}[1]{\widehat{\mathcal{#1}}}
\newcommand{\scrtl}[1]{\widetilde{\mathcal{#1}}}

\ExplSyntaxOn

\cs_new_protected:Nn \bamboo_define:nnnnnN
{
	\cs_new_protected:cpx { #3 #1 #4 } { \exp_not:N #5{#6{#2}} }
}

\int_step_inline:nnn { `A } { `Z }
{
	\bamboo_define:nnnnnN
	{ \char_generate:nn { #1 } { 11 } }
	{ \char_generate:nn { #1 } { 11 } }
	{ bl }
	{ }
	{ \mymathbold }
	\use:n
	
	\bamboo_define:nnnnnN
	{ \char_generate:nn { #1 } { 11 } }
	{ \char_generate:nn { #1 } { 11 } }
	{ scr }
	{ }
	{ \mathcal }
	\use:n
	
	\bamboo_define:nnnnnN
	{ \char_generate:nn { #1 } { 11 } }
	{ \char_generate:nn { #1 } { 11 } }
	{ }
	{ hat }
	{ \widehat }
	\use:n
	
	\bamboo_define:nnnnnN
	{ \char_generate:nn { #1 } { 11 } }
	{ \char_generate:nn { #1 } { 11 } }
	{ scr }
	{ hat }
	{ \scrhat }
	\use:n
	
	\bamboo_define:nnnnnN
	{ \char_generate:nn { #1 } { 11 } }
	{ \char_generate:nn { #1 } { 11 } }
	{ }
	{ br }
	{ \overline }
	\use:n
	
	\bamboo_define:nnnnnN
	{ \char_generate:nn { #1 } { 11 } }
	{ \char_generate:nn { #1 } { 11 } }
	{ }
	{ tl }
	{ \widetilde }
	\use:n
	
	\bamboo_define:nnnnnN
	{ \char_generate:nn { #1 } { 11 } }
	{ \char_generate:nn { #1 } { 11 } }
	{ scr }
	{ br }
	{ \scrbar }
	\use:n
	
	\bamboo_define:nnnnnN
	{ \char_generate:nn { #1 } { 11 } }
	{ \char_generate:nn { #1 } { 11 } }
	{ scr }
	{ tl }
	{ \scrtl }
	\use:n
	
	\bamboo_define:nnnnnN
	{ \char_generate:nn { #1 } { 11 } }
	{ \char_generate:nn { #1 } { 11 } }
	{ }
	{ dt }
	{ \dot}
	\use:n
}
\int_step_inline:nnn { `a } { `z }
{
	\bamboo_define:nnnnnN
	{ \char_generate:nn { #1 } { 11 } }
	{ \char_generate:nn { #1 } { 11 } }
	{ bl }
	{ }
	{ \mymathbold }
	\use:n
	
	\bamboo_define:nnnnnN
	{ \char_generate:nn { #1 } { 11 } }
	{ \char_generate:nn { #1 } { 11 } }
	{ }
	{ hat }
	{ \hat }
	\use:n
	
	\bamboo_define:nnnnnN
	{ \char_generate:nn { #1 } { 11 } }
	{ \char_generate:nn { #1 } { 11 } }
	{ }
	{ br }
	{ \bar }
	\use:n
	
	\bamboo_define:nnnnnN
	{ \char_generate:nn { #1 } { 11 } }
	{ \char_generate:nn { #1 } { 11 } }
	{ }
	{ tl }
	{ \tilde }
	\use:n                            
	
	\bamboo_define:nnnnnN
	{ \char_generate:nn { #1 } { 11 } }
	{ \char_generate:nn { #1 } { 11 } }
	{ }
	{ dt }
	{ \dot}
	\use:n
}
\clist_map_inline:nn
{
	Gamma,Delta,Theta,Lambda,Xi,Pi,Sigma,Phi,Psi,Omega
}
{
	\bamboo_define:nnnnnN
	{ #1 }
	{ #1 }
	{ bl }
	{ }
	{ \mymathbold }
	\use:c
	
	\bamboo_define:nnnnnN
	{ #1 }
	{ #1 }
	{ op }
	{ }
	{ \op }
	\use:c
	
	\bamboo_define:nnnnnN
	{ #1 }
	{ #1 }
	{ scr }
	{ }
	{ \mathcal }
	\use:c
	
	\bamboo_define:nnnnnN
	{ #1 }
	{ #1 }
	{ }
	{ hat }
	{ \widehat }
	\use:c
	
	\bamboo_define:nnnnnN
	{ #1 }
	{ #1 }
	{ }
	{ br }
	{ \overline }
	\use:c
	
	\bamboo_define:nnnnnN
	{ #1 }
	{ #1 }
	{ }
	{ tl }
	{ \widetilde }
	\use:c
	
	\bamboo_define:nnnnnN
	{ #1 }
	{ #1 }
	{ }
	{ dt }
	{ \dot}
	\use:c
	
}
\clist_map_inline:nn
{
	alpha,beta,gamma,delta,epsilon,zeta,eta,theta,iota,kappa,
	lambda,mu,nu,xi,pi,rho,sigma,tau,phi,chi,psi,omega
}
{
	\bamboo_define:nnnnnN
	{ #1 }
	{ #1 }
	{ bl }
	{ }
	{ \mymathbold }
	\use:c
	
	\bamboo_define:nnnnnN
	{ #1 }
	{ #1 }
	{ }
	{ hat }
	{ \hat }
	\use:c
	
	\bamboo_define:nnnnnN
	{ #1 }
	{ #1 }
	{ }
	{ br }
	{ \bar }
	\use:c
	
	\bamboo_define:nnnnnN
	{ #1 }
	{ #1 }
	{ }
	{ tl }
	{ \tilde }
	\use:c
	
	\bamboo_define:nnnnnN
	{ #1 }
	{ #1 }
	{ }
	{ dt }
	{ \dot}
	\use:c
}

\ExplSyntaxOff

\DeclareMathOperator{\E}{\mathbf{E}}

\newcommand{\rank}{\operatorname{rank}}

\newcommand{\tr}{\operatorname{tr}}

\newcommand{\diag}{\operatorname{diag}}

\DeclarePairedDelimiter{\norm}{\lVert}{\rVert}

\DeclarePairedDelimiter{\abs}{\lvert}{\rvert}
\DeclarePairedDelimiter{\floor}{\lfloor}{\rfloor}
\DeclarePairedDelimiter{\ceil}{\lceil}{\rceil}
\DeclarePairedDelimiter{\braces}{\{}{\}}
\DeclarePairedDelimiter{\parens}{(}{)}
\DeclarePairedDelimiter{\brackets}{[}{]}
\DeclarePairedDelimiterX{\ip}[2]{\langle}{\rangle}{#1,#2}

\DeclarePairedDelimiterXPP{\normsub}[2]{}{\lVert}{\rVert}{_{#2}}{#1}
\DeclarePairedDelimiterXPP{\ipsub}[3]{}{\langle}{\rangle}{_{#3}}{#1,#2}

\DeclarePairedDelimiterXPP{\ipHS}[2]{}{\langle}{\rangle}{_{\mathrm{HS}}}{#1, #2}
\DeclarePairedDelimiterXPP{\normHS}[1]{}{\lVert}{\rVert}{_{\mathrm{HS}}}{#1}

\DeclarePairedDelimiterXPP{\ipF}[2]{}{\langle}{\rangle}{_{\mathrm{F}}}{#1, #2}
\DeclarePairedDelimiterXPP{\normF}[1]{}{\lVert}{\rVert}{_{\mathrm{F}}}{#1}

\DeclarePairedDelimiterXPP{\dkl}[2]{\operatorname{D_{KL}}}{(}{)}{}{#1 \: \delimsize\Vert \: #2}

\DeclarePairedDelimiterXPP{\restr}[2]{}{{}}{\vert}{_{#2}}{#1}

\newcommand{\R}{\mathbf{R}}

\newcommand{\range}{\operatorname{range}}

\newcommand{\negqquad}{\mspace{-36mu}}

\theoremstyle{plain}%
\newtheorem{theorem}{Theorem}
\newtheorem{lemma}{Lemma}%

\theoremstyle{definition}%
\Crefname{assumption}{Assumption}{Assumptions}
\newtheorem{example}{Example}%
\newtheorem{remark}{Remark}%

\DeclarePairedDelimiterXPP{\opnorm}[1]{}{\lVert}{\rVert}{_{\mathrm{op}}}{#1}
\DeclarePairedDelimiterXPP{\nucnorm}[1]{}{\lVert}{\rVert}{_{\mathrm{nuc}}}{#1}

\newcommand{\transpose}{^\top\! } 

\newcommand{\stquad}{\quad\text{s.t.}\quad}
\newcommand{\stqquad}{\qquad\text{s.t.}\qquad}
\newcommand{\PT}{\scrP_{\scrT}}
\newcommand{\PTp}{\scrP_{\scrT^\perp}}

\newcommand{\symms}{\mathbf{S}}

\newcommand{\T}{\scrT}
\newcommand{\Tp}{\scrT^\perp}
\newcommand{\Tst}{\scrT_*}
\newcommand{\Tstp}{\Tst^\perp}
\newcommand{\PTst}{\scrP_{\Tst}}
\newcommand{\PTstp}{\scrP_{\Tstp}}

\newcommand{\Ust}{U_*}
\newcommand{\Vst}{V_*}
\newcommand{\Ustbr}{\Ubr_*}
\newcommand{\Vstbr}{\Vbr_*}
\newcommand{\PU}{P_U}
\newcommand{\PUp}{P_U^\perp}
\newcommand{\PV}{P_V}
\newcommand{\PVp}{P_V^\perp}
\newcommand{\PUst}{P_{\Ustbr}}
\newcommand{\PUstp}{P_{\Ustbr}^\perp}
\newcommand{\PVst}{P_{\Vstbr}}
\newcommand{\PVstp}{P_{\Vstbr}^\perp}

\newcommand{\Mp}{M_\perp}

\newcommand{\rstar}{r_*}
\newcommand{\rmax}{r_{\mathrm{max}}}
\newcommand{\Mst}{M_*}
\newcommand{\Mopt}{M_{\mathrm{opt}}}
\newcommand{\Up}{U_\perp}
\newcommand{\Vp}{V_\perp}
\newcommand{\Qp}{Q_\perp}
\newcommand{\PQ}{P_Q}
\newcommand{\PQp}{P_{\Qp}}

\newcommand{\Wp}{W_\perp}

\newcommand{\cp}{c_\perp}
\newcommand{\ap}{a_\perp}
\newcommand{\kapinv}{\kappa^{-1}}
\newcommand{\mueff}{\mu_{\mathrm{eff}}}
\newcommand{\deltac}{\delta_{\mathrm{crit}}}

\newcommand{\Uspace}{\R^{d \times r}}
\newcommand{\Mspace}{\R^{d_1 \times d_2}}
\newcommand{\Mspaced}{\R^{d \times d}}
\newcommand{\Lspace}{\R^{d_1 \times r}}
\newcommand{\Rspace}{\R^{d_2 \times r}}
\newcommand{\Lstspace}{\R^{d_1 \times \rstar}}
\newcommand{\Rstspace}{\R^{d_2 \times \rstar}}

\newcommand{\errmat}{H}
\newcommand{\genmat}{E}
\newcommand{\genrk}{k}
\newcommand{\genrkk}{\ell}
\newcommand{\HT}{\errmat_{\T}}
\newcommand{\HTp}{\errmat_{\Tp}}
\newcommand{\HTst}{\errmat_{\Tst}}
\newcommand{\HTstp}{\errmat_{\Tstp}}
\newcommand{\rsmall}{r_1}
\newcommand{\rbig}{r_2}
\newcommand{\deltak}{\delta_\genrk}
\newcommand{\PLlong}{Polyak-Lojasiewicz}

\newcommand{\admemail}{andrew.mcrae@enpc.fr}
\newcommand{\ryzemail}{ryz@illinois.edu}
\begin{document}
	\title{Sharp recovery and landscape guarantees\\ for the nonconvex matrix LASSO}
	\author{Andrew D. McRae\thanks{CERMICS, CNRS, ENPC, Institut Polytechnique de Paris, Marne-la-Vallée, France. E-mail: \href{mailto:\admemail}{\admemail}.}
		\and 
		Richard Y. Zhang\thanks{Dept.\ of Electrical and Computer Engineering, University of Illinois Urbana-Champaign, United States. E-mail: \href{mailto:\ryzemail}{\ryzemail}.}
	}

\maketitle

\begin{abstract}
Low-rank matrix recovery can be solved to statistical optimality by convex matrix optimization under the classical assumption of restricted isometry property (RIP).
However, for large problems, the convex formulation is commonly replaced by a smooth rank-constrained factored nonconvex problem
for which algorithmic theory typically only guarantees convergence to second-order critical points.
In this paper, we develop a sharp and statistically optimal theory for second-order critical points of the factored nonconvex matrix LASSO (nuclear-norm--regularized least-squares estimator) under RIP
with particular emphasis on the overparametrized regime
where the search rank $r$ exceeds the ground-truth rank $r_*$.
Our recovery error bounds reveal the precise role of nuclear norm regularization, interpolating between the classical convex rate and known rates for the unregularized nonconvex problem.
Complementing this positive result, we give examples showing that, contrary to popular belief, rank overparametrization does not always improve the optimization landscape even under RIP.
This negative result raises questions about the fundamental statistical recovery capability of rank-constrained nonconvex approaches in comparison to convex approaches which have worse computational scaling.
All of our results generalize to arbitrary convex functions with nuclear-norm regularization under restricted strong convexity and smoothness.
In particular, we give sharp conditions under which second-order critical points of the nonconvex problem either (1) approximately recover low-rank approximate minima of the convex problem or (2) exactly recover a low-rank global optimum if one exists.
\end{abstract}

\section{Introduction}
\label{sec:intro}
Low-rank matrix recovery is a basic task in
statistics, signal processing, and machine learning ~\cite{Davenport2016,Hu2022}.
In the linear version of the problem, we seek to recover an unknown matrix
$\Mst \in \Mspace$ of rank $\rstar \ll \min\{d_{1},d_{2}\}$
from $n$ noisy linear measurements of the form
\begin{equation}
	b = \scrA(\Mst)+\xi \in\R^{n}, \qquad \text{where} \qquad (\scrA(M))_{i}=\ip{A_i}{M}, \label{eq:ms_model}
\end{equation}
$A_1, \dots, A_n \in \Mspace$ are known measurement matrices,
and $\ip{\cdot}{\cdot}$ denotes the standard (trace) inner product on matrices.
This problem is NP-hard in general
but has been shown to be tractable in many settings
via convex optimization heuristics~\cite{Fazel2002,Recht2010}.
In particular, one popular convex approach is the matrix LASSO~\cite{Ma2011a,Rohde2011,Candes2011b,Negahban2011}:
\begin{equation}
	\min_{M\in\Mspace}~\frac{1}{2}\norm{\scrA(M)-b}^{2}+\lambda\nucnorm M,\label{eq:lasso} 
\end{equation}
where $\norm{\cdot}$ denotes the ordinary Euclidean norm on $\R^n$,
and $\nucnorm{\cdot}$ denotes the nuclear norm (the sum of
the singular values).
Later extensions of the matrix LASSO to nonlinear measurements have also been shown to be effective~\cite{Lafond2015}.

Theoretically explaining the tractability of low-rank matrix recovery
requires making assumptions on the measurements.
A classical assumption is the matrix \emph{restricted isometry property} (RIP), introduced in \cite{Recht2010},
which requires the linear measurement operator $\scrA$ to
preserve approximately the Euclidean norms of low-rank matrices.
We defer a more
general nonlinear version to \Cref{sec:res_gen}.
We say that $\scrA$ has $(\genrk,\deltak)$-RIP if it satisfies
\begin{equation}
	\rank(\genmat)\leq \genrk \quad \implies \quad
	(1-\deltak)\normF{\genmat}^{2}\leq \norm{\scrA(\genmat)}^{2}\leq (1+\deltak) \normF{\genmat}^{2}, \label{eq:rip}
\end{equation}
where $\normF{\cdot}$ is the matrix Frobenius (elementwise Euclidean) norm.
For example, when the measurement matrices $A_1, \dots, A_n$ are independent random matrices whose entries are independent standard normal random variables,
$\scrA$ has, with high probability and with appropriate rescaling, $(\genrk, \deltak)$-RIP if the number of measurements satisfies $n \gtrsim \deltak^{-2} \genrk (d_1 + d_2)$~\cite{Candes2011b,Davenport2016}.

Low-rank matrix recovery under RIP, Gaussian measurements, or other similar assumptions is commonly known as \emph{matrix sensing}.
Although RIP is a restrictive assumption for many applications,
it is a valuable theoretical starting point for more realistic settings,
such as matrix completion or phase retrieval \cite{Davenport2016}.

The matrix LASSO (\ref{eq:lasso}) was shown in \cite{Candes2011b,Rohde2011,Negahban2011}
to achieve statistically optimal recovery under RIP:
\begin{theorem}[{\cite[Thms.~2.4 \& 2.5]{Candes2011b}}]
	\label{thm:lasso}If $\mathcal{A}$ satisfies $(\genrk,\deltak)$-RIP
	and 
	\[
		\deltak < \frac{3\sqrt{2}-1}{17}
		\quad \text{for} \quad \genrk \geq 4\rstar,
	\]
	then any solution $\widehat{M}$ with $\lambda \geq 2\opnorm{\scrA^*(\xi)}$
	satisfies 
	\[
		\normF{\Mhat - \Mst} \lesssim \sqrt{\rstar} \lambda.
	\]
	With Gaussian noise, the above error bound for $\lambda\approx\opnorm{\scrA^{*}(\xi)}$
	is statistically optimal within a multiplicative constant. 
\end{theorem}
Subsequent work has improved the RIP constant threshold for different variants of the problem.
In particular, one has the same statistically optimal error bound via convex programming if
$\delta_{2\rstar} < 1/\sqrt{2}$,
and this threshold is sharp.
See \Cref{subsec:relcvx} for further discussion.

The error bound cannot be improved
for the LASSO algorithm due to the shrinkage effect of the nuclear norm. Indeed, consider the case where $\scrA$
is the identity on $\Mspace$, and thus \eqref{eq:lasso} is solved by singular
value soft thresholding.
In this case, even with $\xi = 0$,
$\normF{\Mhat - \Mst} = \sqrt{\rstar} \lambda$ for any $\lambda \leq \sigma_{\rstar}(\Mst)$
(here and throughout this paper, $\sigma_j(M)$ will denote the $j$th singular value of the matrix $M$, where the singular values are numbered in decreasing order).

While statistically optimal and solvable in polynomial time,
the convex formulation is still computationally challenging for large matrices due to the need to form and
manipulate large $d_{1}\times d_{2}$ matrices.
In practice, the matrix LASSO is commonly used with a nonconvex factored
formulation
\begin{equation}
	\min_{\substack{U\in\Lspace \\ V\in\Rspace}}~f_{\lambda}(U,V)
	\coloneqq\frac{1}{2}\norm{\scrA(UV\transpose)-b}^{2} + \lambda \frac{\normF{U}^2 + \normF{V}^2}{2} \label{eq:lasso_fact}
\end{equation}
for some search rank hyperparameter $r$.
This is connected to the convex formulation by the identity
\[
	\nucnorm{M} = \min_{M=UV\transpose}~\frac{\normF{U}^2 + \normF{V}^2}{2}.
\]
The factored problem has $r(d_1 + d_2)$ variables;
this is far smaller than the original $d_1 d_2$ variables if $r \ll \min~\{d_1, d_2\}$ (as will typically be the case).
However, the problem \eqref{eq:lasso_fact} is nonconvex.
In general, scalable local search algorithms such as gradient descent
cannot promise convergence to a global optimum,
only to a \emph{second-order critical point}, that is, a point $(U, V)$ with zero gradient and positive semidefinite Hessian:
\[
	\nabla f_{\lambda}(U,V) =0 \qquad \text{and} \qquad\nabla^{2}f_{\lambda}(U,V) \succeq 0.
\]
See, for example, \cite{Cartis2012,Lee2019b} for relevant algorithmic guarantees.
This raises the central question of this paper.

\paragraph{Question:} Under what RIP conditions do second-order points of the nonconvex matrix LASSO satisfy the same optimal error bound as the convex matrix LASSO?

This question becomes especially interesting once one accounts for \emph{rank overparametrization}.
If the true rank $\rstar$ were known, one would naturally set the search rank to $r=\rstar$.
In practice, however, $\rstar$ is often unknown, so, choosing $r$ conservatively, one often has $r>\rstar$.

Even when a good estimate of $\rstar$ is available,
practitioners often still choose a larger value of $r$.
This is motivated by the belief that overparametrization leads to a better optimization landscape; see, for example, \cite{Gunasekar2017,Li2018} for further discussion.
This has been backed up by some theoretical landscape guarantees under RIP \cite{Zhang2025a}; see \Cref{sec:relwork} for further discussion and references.
There are much stronger theoretical results on the benefits of overparametrization for problems with different structure such as linear semidefinite programs (see, e.g., \cite{Ling2026} and the references therein) and phase retrieval \cite{McRae2026a}.

However, overparametrization increases the model complexity and potentially leads to additional sensitivity to noise (``overfitting'').
Thus, as in the rank-unconstrained convex case, regularization, controlled
by the parameter $\lambda$, becomes essential to maintain optimal performance with respect to noise.

As we review in \Cref{sec:relwork}, the existing literature on low-rank matrix recovery via the nonconvex problem \eqref{eq:lasso_fact}
only covers only two partial regimes:
\begin{enumerate}
	\item results that optionally consider overparametrization ($r > \rstar$) but without regularization, that is, $\lambda=0$ (e.g., \cite{Zhang2025a}); and
	\item results with regularization ($\lambda > 0$) but which do not study the effect of overparametrization (e.g., \cite{McRae2026,Agterberg2025}).
\end{enumerate}
By contrast, the joint effect of overparametrization
and regularization has seen limited study.
The recent work \cite{Ouyang2025} gives sharp results from an optimization perspective (recovering a low-rank global optimum if one exists)
but does not apply directly under assumptions like those of \Cref{thm:lasso}.
This paper fills that gap and, in doing
so, establishes the full nonconvex counterpart of \Cref{thm:lasso} for matrix LASSO.

For low-rank matrix recovery from linear RIP measurements, our main result is the following
(for $x \in \R$, we write $x_+ \coloneqq \max~\{x, 0\}$):
\begin{theorem}
	\label{thm:matsens}
	For $r \geq \rstar \geq 1$,
	assume the model \eqref{eq:ms_model} for some rank-$\rstar$ $\Mst \in \Mspace$,
	and	suppose $\scrA \colon \Mspace \to \R^n$ has $(\genrk, \deltak)$-RIP with
	\begin{equation}
		\label{eq:deltacrit}
		\deltak < \deltac \coloneqq \frac{1}{1+\sqrt{\rstar/r}} \quad \text{for} \quad \genrk \geq r + \rstar.
	\end{equation}
	Then, for any $\lambda \geq 0$,
	any second-order critical point $(U, V)$ of \eqref{eq:lasso_fact}
	satisfies
	\begin{gather}
		\normF{UV\transpose-\Mst}\le\frac{6\sqrt{\rstar}\lambda+\sqrt{r+\rstar}(\opnorm{\scrA^{*}(\xi)}-\lambda)_{+}}{\deltac - \deltak}. \label{eq:ms_errbd}
	\end{gather}
	The dependence on $r$ and $\rstar$ is sharp in the thresholds for $\deltak$ and $\genrk$ in (\ref{eq:deltacrit}) and is also sharp in the error bound (\ref{eq:ms_errbd}).
	In particular, $(k, \delta_k)$-RIP for $k < r + \rstar$ does not, in general, suffice to ensure the absence of spurious local optima.
\end{theorem}
The positive side of the result follows as a corollary of the more general \Cref{thm:noisy} in \Cref{sec:res_recov}, established under a nonlinear generalization of RIP.
The ``sharpness'' is shown in two brief examples below and in the fuller counterexamples of \Cref{sec:res_counter}.

\Cref{thm:matsens} and its extensions in \Cref{sec:res_gen} give conditions under which the nonconvex landscapes of \eqref{eq:lasso_fact} and other similar problems are guaranteed to be \emph{benign}:
that is, all second-order critical points (and, in particular, all local optima)
are ``good'' in the sense of being accurate estimates of $\Mst$ or of being globally optimal.

We now discuss in more detail the various aspects and implications of \Cref{thm:matsens}.

\paragraph{Error bound:}
Our \Cref{thm:matsens} unifies, sharpens, and interpolates the existing recovery
guarantees in the literature. The recovery error bound \eqref{eq:ms_errbd} has two terms:
\[
	\normF{UV\transpose-\Mst} \quad \lesssim \qquad  \underbrace{\sqrt{\rstar}\lambda}_{\mathclap{\text{shrinkage error}}}\qquad + \qquad \underbrace{\sqrt{r+\rstar}(\opnorm{\scrA^{*}(\xi)}-\lambda)_{+}}_{\mathclap{\text{excess noise error}}}.
\]
The first term is the classical error of the
convex matrix LASSO (see \Cref{thm:lasso}) due to the shrinkage effect of nuclear norm regularization, while
the second term quantifies the additional error due to noise when the regularization is insufficient.
Several consequences are immediate: 
\begin{itemize}
	\item If $\lambda \geq \opnorm{\scrA^{*}(\xi)}$, the ``excess noise'' term vanishes,
	and every second-order point achieves, independently of the search rank $r$, the same error rate as in the convex case (\Cref{thm:lasso}).
	As in the convex case, with noise-calibrated regularization $\lambda \approx \opnorm{\scrA^*(\xi)}$,
	this is statistically optimal,
	so rank overparametrization carries no statistical penalty in terms of recovery error.
	\item If $\lambda=0$, the shrinkage term vanishes, and the bound reduces to 
	\[
		\normF{UV\transpose-\Mst} \lesssim \sqrt{r+\rstar} \opnorm{\scrA^{*}(\xi)} \approx \sqrt{r} \opnorm{\scrA^{*}(\xi)}.
	\]
	We then obtain an improved version of the result of \cite{Zhang2025a} (see \Cref{subsec:relrip} for further discussion and comparison).
	The dependence on $r$ is unavoidable (see \Cref{ex:error} below).
	
	\item When $0<\lambda<\opnorm{\scrA^{*}(\xi)}$, the bound interpolates between the two previous regimes.
\end{itemize}
The resulting picture is simple.
Overparametrization can indeed increase sensitivity to noise
but only when we have insufficiently regularized the problem (that is, $\lambda < \opnorm{\scrA^*(\xi)}$).
Statistically optimal recovery can be obtained either
\begin{enumerate}
	\item by using a not-too-overparametrized search rank $r=O(\rstar)$
	or 
	\item by, as with the convex LASSO, using an appropriately tuned regularization $\lambda \approx \opnorm{\scrA^{*}(\xi)}$.
\end{enumerate}

In the regime $\lambda \leq \opnorm{\scrA^*(\xi)}$ (for the case $\lambda \geq \opnorm{\scrA^*(\xi)}$, see the example given after \Cref{thm:lasso}),
the error bound is indeed optimal within a multiplicative constant, as shown by the following example:
\begin{example}
	\label{ex:error}
	Take $\scrA$ to be the identity on $\Mspace$.
	For $r \geq \rstar \geq 1$ and $d_1, d_2 \geq r + \rstar$,
	let $\Mtl, \Mst \in \Mspace$ be matrices with orthogonal row and column spaces satisfying $\rank(\Mtl) = r$, $\rank(\Mst) = \rstar$, and $\sigma_1(\Mtl) = \sigma_r(\Mtl) = \sigma_1(\Mst) = \sigma_{\rstar}(\Mst) = 1$.

	Set $b = \scrA(\Mtl) = \Mtl$;
	then $\opnorm{\scrA^*(\xi)} = \opnorm{\Mtl - \Mst} = 1$,
	and, for $\lambda \geq 0$, the LASSO problem
	\[
		\min_{M \in \Mspace}~\frac{1}{2} \norm{\scrA(M) - b}^2 + \lambda \nucnorm{M} = \min_{M \in \Mspace}~\frac{1}{2} \normF{M - \Mtl}^2 + \lambda \nucnorm{M},
	\]
	has the unique solution $\Mhat = (1 - \lambda)_+ \Mtl$ (as all nonzero singular values of $\Mtl$ are 1).
	As $\rank(\Mhat) \leq r$, all global optima $(U, V)$ of the nonconvex problem \eqref{eq:lasso_fact} must satisfy $U V\transpose = \Mhat$.
	
	We can directly calculate, for $\lambda \leq \opnorm{\scrA^*(\xi)} = 1$,
	\begin{align*}
			\normF{\Mhat - \Mst}
			&= \sqrt{ \rstar + (1 - \lambda)_+^2 r} \\
			&\geq \frac{1}{\sqrt{2}}( \sqrt{\rstar} + \sqrt{r}(1 - \lambda)_+ ) \\
			&\geq \frac{1}{\sqrt{2}}( \sqrt{\rstar} \lambda + \sqrt{r + \rstar}(1 - \lambda)_+ ) \\
			&= \frac{1}{\sqrt{2}} \parens{ \sqrt{\rstar} \lambda + \sqrt{r + \rstar} (\opnorm{\scrA^*(\xi)} - \lambda)_+ }.
		\end{align*}
	By rescaling, we can extend this example to any value of $\opnorm{\scrA^*(\xi)} \geq 0$.
\end{example}

\paragraph{RIP condition:}
\Cref{thm:matsens} also fully clarifies how the landscape evolves under rank overparametrization.
By the definition of RIP, the constants $\deltak$ are monotone in $\genrk$,
that is, $\delta_{1} \le \delta_{2} \le \delta_{3} \le \cdots$.
The sharp dependence on both $\genrk$ and $\deltak$ shows that
the potential benefits of overparametrization hinge not merely on the size of a single RIP constant
but also on the profile of $\genrk \mapsto \deltak$.
If this profile is flat, for example if $\delta_{2\rstar} = \delta_{2\rstar+1} = \delta_{2\rstar+2} = \cdots$,
then increasing the search rank $r$ does indeed improve the landscape.
This is consistent with popular belief and practice (see discussion above).

If, on the other hand, $\deltak$
grows too quickly with $\genrk$, then increasing the search rank can instead
worsen the landscape.
The following concrete counterexample (for the case $\lambda = 0$), which expands on \cite[Ex.~4.4]{Zhang2025a},
exploits exactly this mechanism, producing spurious local optima for large search rank $r$ even when the landscape is benign for smaller $r$.
More general counterexamples, which handle the case $\lambda > 0$,
are given in \Cref{sec:res_counter}.
\begin{example}
	\label{ex:overpbad}
	\newcommand{\Uspur}{U_{\mathrm{sp}}}
	\newcommand{\Vspur}{V_{\mathrm{sp}}}
	\newcommand{\Ugt}{P}
	\newcommand{\Vgt}{Q}
	\newcommand{\Uorth}{P_\perp}
	\newcommand{\Vorth}{Q_\perp}
	\newcommand{\cspur}{c_{\mathrm{sp}}}
	\newcommand{\rspur}{r_{\mathrm{sp}}}
	\newcommand{\kapcr}{\kappa_{\mathrm{crit}}}
	\newcommand{\kappak}{\kappa_\genrk}
	\newcommand{\kapspur}{\kappa_{\mathrm{sp}}}
	\newcommand{\Lspspace}{\R^{d_1 \times \rspur}}
	\newcommand{\Rspspace}{\R^{d_2 \times \rspur}}
	
	Instead of the RIP constant $\deltak$, it will be more convenient (to avoid scaling issues) to consider
	the restricted \emph{condition number}
	\[
		\kappak \coloneqq \max~\frac{\norm{\scrA(\genmat_1)}^2}{\norm{\scrA(\genmat_2)}^2} \stquad \normF{\genmat_1} = \normF{\genmat_2} = 1,\ \rank(\genmat_1), \rank(\genmat_2) \leq \genrk.
	\]
	\Cref{thm:matsens} then applies (potentially after rescaling $\scrA$) if and only if
	\begin{equation}
		\label{eq:kappacrit}
		\kappak < \kapcr(r, \rstar) \coloneqq \frac{1 + \deltac(r, \rstar)}{1 - \deltac(r, \rstar)}  = 1 + 2 \sqrt{\frac{r}{\rstar}} \qquad \text{for} \qquad \genrk = r + \rstar.
	\end{equation}
	Indeed, our more general results in \Cref{sec:res_gen} will resemble this (see, in particular, \Cref{rmk:Lequal}).
	
	Now, we construct the example.
	For any arbitrary $\kapspur > 1$, we define a linear operator $\scrA$ with condition number at most $\kapspur$,
	that is, $\kappak \leq \kapspur$ for all $\genrk \geq 1$.
	We achieve this by requiring
	\[
		\norm{\scrA(\genmat)}^2 = \normF{\genmat}^2 - (1 - \kapspur^{-1}) \ip{G}{\genmat}^2 \quad \text{for} \quad \genmat \in \Mspace
	\]
	for some $G \in \Mspace$ that we will choose such that $\normF{G} = 1$.
	For example, we can choose $\scrA \colon \Mspace \to \Mspace$ to be a rank-1 perturbation of the identity operator on $\Mspace$:
	\[
		\scrA(\genmat) = \genmat - (1 - \kapspur^{-1/2}) \ip{G}{\genmat}G.
	\]
	Fix $\rspur \geq \rstar \geq 1$ and $d_1, d_2 \geq \rspur + \rstar$.
	Choose $\Ugt \in \Lstspace$, $\Uorth \in \Lspspace$, $\Vgt \in \Rstspace$, and $\Vorth \in \Rspspace$
	such that the matrices $[\Ugt\ \Uorth] \in \R^{d_1 \times (\rspur + \rstar)}$ and $[\Vgt\ \Vorth] \in \R^{d_2 \times (\rspur + \rstar)}$ both have orthonormal columns.
	We will choose
	\[
		\Mst \coloneqq \Ugt \Vgt\transpose \qquad \text{and} \qquad G \coloneqq \frac{1}{2 \sqrt{\rstar}} \Ugt \Vgt\transpose - \frac{1}{2 \sqrt{\rspur}} \Uorth \Vorth\transpose.
	\]
	Crucially, this construction ensures $\kappak = \kapspur$ for all $\genrk \ge \rspur + \rstar$.
	
	Next, consider the (unregularized) nonconvex landscape of
	\[
		f(U, V) \coloneqq \frac{1}{2} \norm{\scrA(U V\transpose - \Mst)}^2 \qquad \text{for} \qquad U \in \Lspspace,\ V \in \Rspspace.
	\]
	As $\scrA$ is injective, any global optimum must satisfy $U V\transpose = \Mst$.
	However, direct calculation (or numerical verification) reveals that the point $(\Uspur, \Vspur)$, defined by
	\[
		\Uspur = \cspur \Uorth, \qquad \Vspur = \cspur \Vorth, \qquad \cspur^2 = \frac{\kapspur - 1}{\kapspur + 1} \sqrt{\frac{\rstar}{\rspur}}
	\]
	is a first-order critical point, that is, $\nabla f(\Uspur, \Vspur) = 0$.
	Furthermore, the Hessian is positive semidefinite (i.e., $\nabla^2 f(\Uspur, \Vspur) \succeq 0$)
	if and only if $\kappa_{\rspur + \rstar} = \kapspur \geq \kapcr(\rspur, \rstar)$,
	where $\kapcr$ is defined in \eqref{eq:kappacrit}.
	In fact, when $\kapspur > \kapcr(\rspur, \rstar)$ strictly, $(\Uspur, \Vspur)$ is a spurious local optimum;
	in this case, one can verify numerically that local optimization algorithms with random initialization will occasionally converge to this point (modulo problem symmetries) and thereby fail to find a global optimum.
	The construction up to this point mirrors that of \cite[Ex.~4.4]{Zhang2025a}
	and shows that the threshold $\kapcr$ in \eqref{eq:kappacrit} and, equivalently, the threshold $\deltac$ for $\delta_k$ in \Cref{thm:matsens} are tight.
	
	Thus far, however, we have only considered the restricted condition number $\kappak$ for $\genrk \geq \rspur + \rstar$.
	What is the condition number for smaller ranks $\genrk < \rspur + \rstar$?
	Critically, for any $r \in [\rstar, \rspur]$, we have
	\[
		\max_{\substack{\genmat \in \Mspace \\ \normF{\genmat} = 1 \\ \rank(\genmat) \leq r + \rstar}}~\ip{G}{\genmat}^2
		= \sigma_1^2(G) + \cdots + \sigma_{r + \rstar}^2(G)
		= \frac{1 + r/\rspur}{2},
	\]
	and, therefore,
	\[
		\kappa_{r + \rstar} = \frac{1}{1 - (1 - \kapspur^{-1}) \frac{1 + r/\rspur}{2} }.
	\]
	Hence we conclude that $\kappa_{r + \rstar} < \kapspur$ for all $r < \rspur$.
	Notably, regardless of how large $\kapspur$ is,
	if $\rstar \leq r < 3 \rspur$,
	then $\kappa_{r + \rstar} < 3 \leq \kapcr(r, \rstar)$.
	In such cases, \Cref{thm:matsens} guarantees a benign landscape with search rank $r$.
	
	In particular, for the exactly parametrized case $r = \rstar$, we have $\kappa_{2\rstar} < 3$, which, up to rescaling, is equivalent to $\delta_{2 \rstar} < 1/2$.
	In this case, convex approaches (see \Cref{subsec:relcvx})
	and the nonconvex problem \eqref{eq:lasso_fact} with $r = \rstar$ (by \Cref{thm:matsens})
	are guaranteed to succeed,
	but an overparametrized version (with $r = \rspur > \rstar$) of the nonconvex problem has a spurious local minimum.
\end{example}
Our more general counterexamples in \Cref{sec:res_counter} below have two additional refinements:
\begin{itemize}
	\item We can adjust the constants in the construction of $\scrA$ for various purposes.
	In particular, we can find counterexamples with arbitrarily small $\delta_{2\rstar} > 0$ yet a spurious local optimum for sufficiently large $r$.
	\item For the case $\lambda > 0$, we carefully choose $b = \scrA(\Mst) + \xi$ such that, among other things, $\opnorm{\scrA^*(\xi)} = \lambda$, and $\Mst$ remains a global optimum of the convex problem \eqref{eq:lasso}.
\end{itemize}

\paragraph{Empirical results for matrix sensing with Gaussian measurements:}
\begin{figure}[t]
	\centering
	\includegraphics[width=0.95\linewidth]{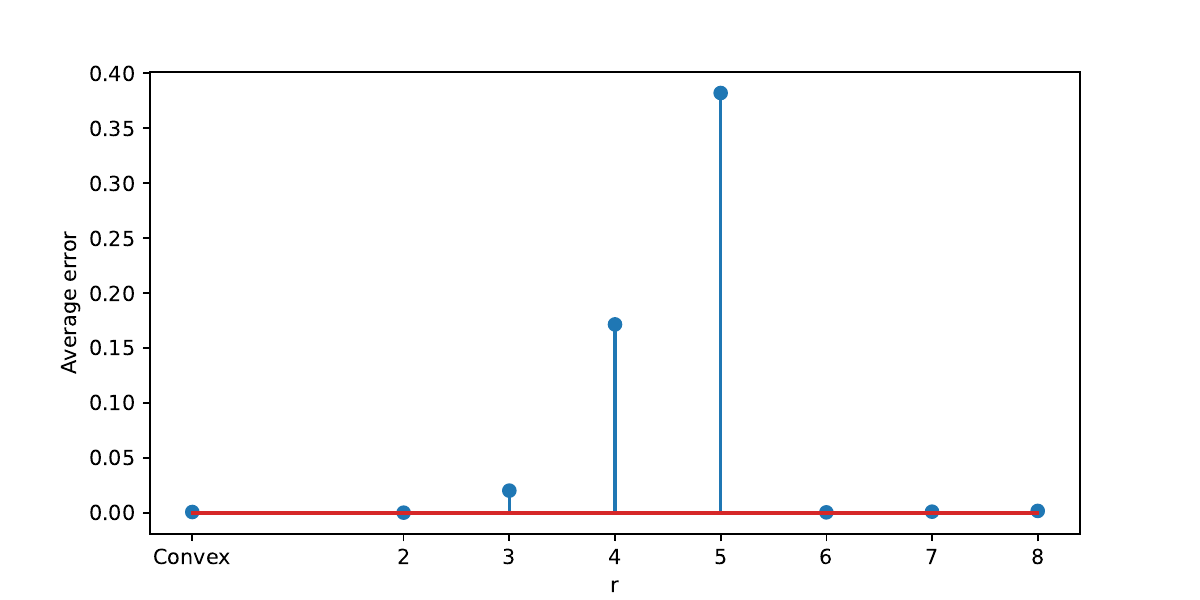}
	\caption{%
		Average error of $\normF{U V\transpose - \Mst}$ for local optimization of \eqref{eq:lasso_fact}
		with Gaussian measurements, different values of the search rank parameter $r$, and random initialization of $(U, V)$.
		We used $d_1 = 50$, $d_2 = 51$, $\rstar = 2$, and $n = \ceil{2.35 \rstar (d_1 + d_2)} = 475$ Gaussian measurements.
		All nonzero singular values of $\Mst$ are $1$.
		We chose $\xi = 0$ and $\lambda = 0.0001$ (nonzero to ensure a unique global optimum),
		and $\scrA$ is scaled such that $\E \norm{\scrA(\genmat)}^2 = \normF{\genmat}^2$ for all $\genmat \in \Mspace$.
		We use the Newton conjugate gradient trust-region algorithm \cite{Nocedal2006} as implemented in SciPy.
		The plotted errors are an average across $50$ random draws of $\scrA$.}
	\label{fig:gauss_exps}
\end{figure}
One might reasonably wonder whether the carefully engineered adversarial counterexamples we have just described are really representative of what happens in statistical settings.
Surprisingly, even with Gaussian measurements (a popular idealized statistical model), we indeed empirically observe problem instances for which increasing the search rank can worsen the landscape.
In the experiments described and visualized in \Cref{fig:gauss_exps}, one observes a benign landscape for $r = \rstar$,
but spurious local minima appear for slightly larger values of $r$.
As $r$ continues to increase, the landscape once again becomes benign.
It would be interesting to explore further this phenomenon,
but this is beyond the scope of the present paper, and we leave it to future work.

\paragraph{Computational complexity of low-rank matrix recovery:}
The negative side of our results raises a deep question:
is there a fundemental difference between the statistical recovery capability of convex programming
and that of more computationally scalable nonconvex approaches?
For example, for a problem with $1/2 < \delta_{2\rstar} < 1/\sqrt{2}$,
we know that convex programming can solve it (see \Cref{subsec:relcvx}),
but the nonconvex problem \eqref{eq:lasso_fact} may have a spurious local optimum with $r = \rstar$.
A standard response based on prior work such as \cite{Zhang2025a} would be to set $r > \rstar$ to compensate.
Indeed, if $\genrk \mapsto \deltak$ is sufficiently constant for $\genrk > 2 \rstar$, this will succeed
(for example, if $\delta_{7\rstar} \leq 1/\sqrt{2}$, \Cref{thm:matsens} guarantees a benign landscape for $r = 6 \rstar$).
However, our counterexamples call this assumption and hence the entire strategy of overparametrization into question.

Further work is thus needed to understand fully this possible gap between convex and nonconvex approaches.
This includes both (1) refining and extending the class of counterexamples developed in the present work and (2) exploring further (empirically and theoretically) the behavior under Gaussian or related statistical models as suggested by \Cref{fig:gauss_exps}.

\section{Related work}

\label{sec:relwork} 

\subsection{Convex recovery guarantees}

\label{subsec:relcvx}

The sharp statistical limits of RIP matrix sensing are currently
best understood through the \emph{matrix Dantzig selector} introduced by \cite{Candes2011b}:
\[
	\min_{M\in\Mspace}~\nucnorm{M} \stqquad \opnorm{\scrA^*(\scrA(M) - b)} \leq \lambda.
\]
For this estimator,~\cite{Cai2014} established
the sharp RIP threshold 
\[
	\deltak <\delta_* \coloneqq
	\sqrt{1-\frac{\rstar}{\genrk}}, \qquad \genrk \geq \frac{4}{3}\rstar
\]
for statistically optimal recovery $\normF{\Mhat-\Mst} \lesssim \sqrt{\rstar} \lambda$ when $\lambda \geq \opnorm{\scrA^*(\xi)}$ (compare \Cref{thm:lasso}).
In particular, we have $\delta_* = 1/\sqrt{2}$ if $\genrk = 2 \rstar$.

We study the unconstrained matrix LASSO \eqref{eq:lasso} through
its smooth nonconvex low-rank factorization \eqref{eq:lasso_fact}.
As the two estimators are structurally parallel
(see \cite{Candes2011b} and also \cite{Bickel2009} for the sparse vector case),
it is natural to expect that the sharp threshold $\deltak<\delta_*$ and optimal error bound extend to the matrix LASSO.
The work \cite{Wang2021a} indeed obtained the threshold,
but, to the best of our knowledge, there is currently no result in the literature for the formulation \eqref{eq:lasso} with this sharp RIP threshold \emph{and} a statistically optimal error bound.

\subsection{Generic nonconvex guarantees based on dimension}

\label{subsec:relover}
\newcommand{\rc}{r_{\mathrm{crit}}}

A trivial way to transfer guarantees from the convex matrix LASSO \eqref{eq:lasso} to its nonconvex factored form \eqref{eq:lasso_fact}
is to choose $r \geq \min~\{ d_1, d_2 \}$.
The problems are then equivalent as the implied
constraint $\rank(M) \leq r$ becomes vacuous.
Furthermore, the factored problem \eqref{eq:lasso_fact} has a benign landscape
in that every second-order critical point yields a first-order critical point (and thus by convexity a global optimum) of \eqref{eq:lasso}.
This follows from \cite[Thm.~2.3(a)]{Ha2020}
for $\lambda=0$ and \cite[Lem.~4]{McRae2026} for $\lambda > 0$,
noting that the rank constraint in those results is also vacuous.
However, taking $r \geq \min~\{ d_1, d_2 \}$ defeats the original dimension-saving purpose of the factorization.

Empirically, one often succeeds in solving \eqref{eq:lasso} with a search rank $r$ only modestly larger than $\rstar$~\cite{Srebro2004a,Srebro2004,Hastie2015a}.
One way to explain this for $r \ll \min~\{ d_1, d_2 \}$
uses smoothed analysis and dimension-counting.
This argument was first introduced in \cite{Boumal2016a} for factorizations of linear semidefinite programs.
For the factored objective $f_{\lambda}(U,V)$ in (\ref{eq:lasso_fact}),
\cite{Bhojanapalli2018} proved that if 
\[
	r > \rc \coloneqq \floor*{ \frac{\sqrt{8n+1}-1}{2} } \approx \sqrt{2 n},
\]
then, after adding a linear perturbation to obtain
\[
	\ftl_\lambda(U,V) \coloneqq f_\lambda(U, V) +
	\ip*{\Xi}{ \begin{bmatrix*} U U\transpose & U V\transpose \\ V U\transpose & V V\transpose  \end{bmatrix*}},
\]
for almost every symmetric $\Xi \in \R^{(d_1 + d_2) \times (d_1 + d_2)}$,
every second-order critical point of $\ftl_\lambda$ is globally
optimal.
Since, often, $n = O(d_1 + d_2)$, the threshold $\rc \approx \sqrt{2n}$
can be far smaller than $\min~\{d_1, d_2\}$.
The guarantee is also general, not relying on any additional structure of the measurement map $\scrA$.
However, this result has some limitations.
First, $\rc$ can still be considerably larger
than the true rank $\rstar$.
Second, the linear perturbation is artificial and not typically used in practice.
However, \cite{Bhojanapalli2018} showed that this perturbation is not optional;
without it, spurious second-order critical points can persist even for $r > \rc$.

\subsection{Overparametrized nonconvex guarantees based on RIP}

\label{subsec:relrip}
To obtain stronger guarantees, we must impose additional structure
on the measurement map $\scrA$.
We focus, in this paper, on the restricted isometry property (RIP) defined in \eqref{eq:rip}.
Indeed, the strongest and most comparable existing landscape results we are aware of make this (or a similar) assumption.
There has been some work on nonconvex approaches to other problems like matrix completion and phase retrieval (where, typically, RIP does not hold),
but the theoretical results are weaker or difficult to compare, so we do not review them here.

More broadly, two paradigms have emerged in the theoretical literature on nonconvex optimization under RIP.
The \emph{algorithmic} paradigm studies a particular algorithm.
These results have the strength of providing explicit convergence rates.
On the other hand, their assumptions are often suboptimal in terms of RIP constant (or sample complexity),
and the algorithms, to aid theoretical analysis, are often simplified and/or contrived compared to what one would use in practice.
By contrast, the \emph{landscape} paradigm identifies conditions under which the entire nonconvex optimization landscape is benign,
and hence they are applicable with general algorithms.
Our analysis is part of this second paradigm.

Within the algorithmic paradigm,
one popular class of algorithm consists of a sufficiently accurate initialization (e.g., by a spectral method) followed by a specific local descent algorithm such as gradient descent.
However, theoretical analyses of these algorithms often assume that we know the correct search rank exactly,
and the results are quite sensitive to the condition number of the optimum.
The technical reason for this is that, in the overparametrized regime, the (local) \PLlong{} inequality (essentially local strong convexity modulo problem symmetries) fails.
Some attempts have been made recently to overcome this with carefully-designed algorithms: see, for example, \cite{Diaz2025} and the references therein.
For further information and references on this class of results, see the overview \cite{Chi2019} and the more recent work \cite{Stoeger2024}.

More relevant to our study of overparametrization is another class of algorithm which has recently seen rising interest:
that of a small, random initialization followed by local descent for a carefully tuned number of iterations \cite{Stoeger2021,Xu2023a,Ma2023c,Soltanolkotabi2025}.
Interestingly, these results are, in some respects, stronger than our (sharp) global landscape results.
Even for search rank $r \gg \rstar$, they only require $(\genrk, \deltak)$-RIP for $\genrk \approx \rstar$.
Hence it is plausible that there are problem instances with spurious local optima when $r \gg \rstar$
but for which the just-cited results guarantee that small random initialization and local descent avoid these spurious optima.
On the other hand, the algorithms (small initialization and early stopping) are somewhat contrived in the matrix recovery setting,
and the required RIP constants are severely suboptimal.

We now turn our attention to the literature on nonconvex landscapes.
Most existing landscape results in our setting consider the unregularized case $\lambda = 0$ in \eqref{eq:lasso_fact},
that is, they study the nonconvex landscape of
\[
	f_0(U, V)\coloneqq \frac{1}{2}\norm{ \scrA(UV\transpose)-b}^2.
\]
The next subsection will discuss the results we are aware of in the case $\lambda > 0$.
Furthermore, the earliest landscape guarantees considered only symmetric positive semidefinite (PSD) factorizations of the form $M = UU\transpose$ \cite{Bhojanapalli2016}.

The first work we are aware of showing the landscape benefits (in terms of RIP constant) of overparametrization was an early (2021) arXiv version of \cite{Zhang2025a},
which established, in the noiseless symmetric PSD case, the optimal RIP threshold $\deltac = (1 + \sqrt{\rstar/r})^{-1}$ from \Cref{thm:matsens};
this result also appeared in \cite{Zhang2025c}.
This was extended to include some robustness to noise by \cite{Ma2023b}.
The final version of \cite{Zhang2025a} gave more refined error bounds with noise and approximate criticality.

The extension to asymmetric factorizations of the form $M = U V\transpose$ is more delicate due to potential imbalance between the factors:
for any invertible $G\in\R^{r\times r}$, $f_0(U, V) = f_0(UG, V(G^{-1})\transpose)$.
A common remedy is to add a balancing
regularizer, that is, we minimize instead
\[
	g_\gamma(U, V)\coloneqq f_0(U, V) + \gamma\normF{ U\transpose U-V\transpose V}^2
\]
for some $\gamma > 0$ \cite{Tu2016}.
Landscape results with this regularizer (e.g., \cite{Park2017}) typically then convert the problem to one with a (larger) symmetric factorization. 
This was extended to general convex objective functions (like in \Cref{sec:res_gen} of the present paper) by \cite{Zhu2018a,Zhu2021a}.
None of these works considered the effects of overparametrization.

In the overparametrized case, using this balancing regularizer, the work \cite{Zhang2025a} (in its final version) showed that $g_\gamma$ has a benign landscape when $2 \deltak < \deltac = (1 + \sqrt{\rstar/r})^{-1}$ for $\genrk \geq r + \rstar$.
The additional factor of two comes from translating the problem into the symmetric case but is not tight.
Our results in the present paper improve on this by recovering the optimal threshold $\deltak < \deltac$ (this was already partially done by \cite{Ouyang2025} in the case $\lambda > 0$---see \Cref{subsec:relreg} below).
Achieving this requires a substantial refinement of the proof
strategy.

While convenient for analysis, the balancing regularizer is artificial;
it increases computational cost, appears to have little practical
effect, and is therefore often omitted in practice (see, e.g., \cite[Rem.~7]{Zhu2018a}).
Furthermore, the landscape guarantees require precise tuning of the parameter $\gamma > 0$.
The only global landscape results we are aware of for the factored problem without regularization are \cite{Ha2020,Kim2025a} which require, in our notation, $\delta_{2r} < 1/3$.
Our present paper more generally confirms that regularization is unnecessary for sharp landscape guarantees.

\subsection{Nonconvex guarantees with nuclear-norm regularization}

\label{subsec:relreg}
The factored version of the nuclear-norm penalty in \eqref{eq:lasso_fact}
has long been used in practical matrix optimization
for promoting low-rank solutions \cite{Srebro2004,Hastie2015a}.
However, given the nonconvexity of the problem, there is relatively little theory supporting its effectiveness.
With a symmetric PSD factorization $M=UU\transpose$,
it is straightforward to apply prior benign landscape
theory because the nuclear norm reduces to the linear trace:
$\nucnorm{M} = \tr(M)$.
In the general asymmetric case, analysis is complicated by the fact that the nuclear norm is nonlinear and nonsmooth.

To the best of our knowledge,
all existing work on the nonconvex landscape of \eqref{eq:lasso_fact} with $\lambda > 0$ provides guarantees of \emph{global optimality} assuming that there is a global optimum $\Mopt$ to \eqref{eq:lasso} with $\rank(\Mopt) \leq r$
(hence these results are more comparable to our \Cref{thm:global_main} below than to \Cref{thm:matsens}).
Our results, more generally, provide tight error bounds for a low-rank approximate minimizer as in \Cref{thm:matsens}.
Each of the results cited here also applies to general convex functions as in \Cref{sec:res_gen} below,
but, for clarity, we state them in terms of RIP constants.

The first work we are aware of is \cite{Li2019},
which proved a benign landscape when $\delta_{4r} \leq 1/5$
and, critically, showed that the nuclear-norm penalty ensures balance between the two factors.
The work \cite{McRae2026} improved this to $\delta_{2r} < 1/3$ (this was also later independently done by \cite{Kim2025a})
and, further, connected this to the statistical model \eqref{eq:ms_model},
giving conditions, comparable to those of \Cref{thm:lasso}, under which the problem \eqref{eq:lasso} is guaranteed to have a low-rank solution.
The work \cite{Agterberg2025} gives similar guarantees and additional asymptotic approximations under more restrictive statistical and structural assumptions.

Finally, the recent work \cite{Ouyang2025} achieved the tight benign landscape threshold $\delta_{r + \rstar} < \deltac = (1 + \sqrt{\rstar/r})^{-1}$ assuming there is a global optimum $\Mopt$ with $\rank(\Mopt) = \rstar$.
Their main result is identical to our \Cref{thm:global_main} in the case $\lambda > 0$ and (in our notation) $L = L_2$.

\section{Full results for general convex objective}
\label{sec:res_gen}

We will, more generally, consider nuclear-norm regularized problems of the form
\begin{equation}
	\label{eq:cvx_asym}
	\min_{M \in \Mspace}~\phi(M) + \lambda \nucnorm{M}
\end{equation}
where $\lambda \geq 0$, and $\phi \colon \Mspace \to \R$ can now be any convex and twice-differentiable function.
The nonconvex factored formulation is, for rank hyperparameter $r$,
\begin{equation}
	\label{eq:ncvx_asym}
	\min_{\substack{U \in \Lspace \\ V \in \Rspace}}~f_\lambda(U, V), \qquad \text{where} \qquad f_\lambda(U, V) = \phi(U V\transpose) + \lambda \frac{\normF{U}^2 + \normF{V}^2}{2}.
\end{equation}

We will also consider matrix optimization problems over  $d \times d$ symmetric positive semidefinite (PSD) matrices:
\begin{equation}
	\label{eq:cvx_sym}
	\min_{M \succeq 0}~\phi(M)+ \lambda \tr M.
\end{equation}
Recall that the nuclear norm of a PSD matrix is simply its trace.
The function $\phi$ now only needs to be defined on the set of PSD matrices in $\Mspaced$,
which we denote by $\symms_d^+$ (we will denote by $\symms_d$ the full set of symmetric matrices in $\Mspaced$).
The symmetric Burer-Monteiro factored nonconvex problem is then, for rank hyperparameter $r$,
\begin{equation}
	\label{eq:ncvx_sym}
	\min_{U \in \Uspace}~f_\lambda(U), \qquad \text{where} \qquad f_\lambda(U) = \phi(U U\transpose) + \lambda \normF{U}^2.
\end{equation}
As discussed in \Cref{subsec:relrip},
most state-of-the-art nonconvex landscape results have in fact primarily studied this symmetric PSD problem.

We recover the matrix LASSO discussed in \Cref{sec:intro} by taking quadratic $\phi$ of the form
\begin{equation}
	\label{eq:quadraticphi}
	\phi(M) = \frac{1}{2} \norm{\scrA(M) - b}^2,
\end{equation}
where, for some integer $n \geq 1$,
$\scrA$ is a linear operator from $\Mspace$ (asymmetric case) or $\symms_d$ (symmetric case) to $\R^n$,
and $b \in \R^n$.

\subsection{Assumption: restricted strong convexity and smoothness}
\label{sec:rscs}
For quadratic $\phi$ of the form \eqref{eq:quadraticphi},
we assumed, in \Cref{sec:intro}, that $\scrA$ had the restricted isometry property \eqref{eq:rip}.
We need to extend this notion to more general convex $\phi$.
Many prior works have made similar assumptions (e.g., \cite{Ha2020,Zhu2018a,Li2019,Bi2021}).

For integers $\genrk, \genrkk \geq 1$ and $L \geq \mu \geq 0$,
we say $\phi$ on $\Mspace$ (respectively, $\symms_d^+$) has $(\genrk, \genrkk, \mu, L)$-restricted strong convexity and smoothness if it is twice differentiable and, for all $M, \genmat \in \Mspace$ (resp. $M \in \symms_d^+$, $\genmat \in \symms_d$) with $\rank(M) \leq \genrkk$, $\rank(\genmat) \leq \genrk$,
the Hessian quadratic form of $\phi$ at $M$ applied to $\genmat$ satisfies
\[
	\mu \normF{\genmat}^2 \leq \nabla^2 \phi(M)[\genmat, \genmat] \leq L \normF{\genmat}^2.
\]
We will denote by $\mu^{(\genrkk)}_{\genrk}(\phi)$ and $L^{(\genrkk)}_{\genrk}(\phi)$ the optimal values of $\mu$ and $L$ for which this holds.
In the case $\genrk = \genrkk$,
we write $\mu_{\genrk}(\phi) \coloneqq \mu^{(\genrk)}_{\genrk}(\phi)$ and $L_{\genrk}(\phi) \coloneqq L^{(\genrk)}_{\genrk}(\phi)$.

In the case of quadratic $\phi$ of the form \eqref{eq:quadraticphi},
the integer $\genrkk$ is irrelevant (see the Hessian expression in \eqref{eq:quadphi_hess} below),
and the condition is equivalent to restricted isometry:
for all $\genmat$ of rank at most $\genrk$, we have
\[
	\mu \normF{\genmat}^2 \leq \norm{\scrA(\genmat)}^2 \leq L \normF{\genmat}^2,
\]
and, setting $\delta = \frac{L - \mu}{L + \mu}$,
we have
\[
	(1 - \delta) \normF{\genmat}^2 \leq \frac{2}{L + \mu} \norm{\scrA(\genmat)}^2 \leq (1 + \delta) \normF{\genmat}^2,
\]
which is precisely \eqref{eq:rip} up to rescaling of $\scrA$.

For such quadratic $\phi$, we will sometimes write $\mu_\genrk(\scrA) = \mu_\genrk(\phi)$, $L_\genrk(\scrA) = L_\genrk(\phi)$.

\subsection{Approximate recovery}
\label{sec:res_recov}
Our first result for general $\phi$ generalizes \Cref{thm:matsens},
which was our result for approximate recovery via the matrix LASSO.

We must consider how to generalize the assumptions of \Cref{thm:matsens}.
We have already discussed above how to generalize RIP to restricted strong convexity and smoothness.
The other quantity we must generalize is the ``noise'' represented by the vector $\xi$ in the model \eqref{eq:ms_model}.
In \Cref{thm:matsens} and many other results on matrix sensing (such as \Cref{thm:lasso}), we quantify the noise via $\scrA^*(\xi)$.
Note that, if $\phi$ is the least-squares loss \eqref{eq:quadraticphi},
we have, under the model \eqref{eq:ms_model}, $\scrA^*(\xi) = - \nabla \phi(\Mst)$.
This suggests that $\nabla \phi(\Mst)$ could represent ``noise'' for more general $\phi$.
Indeed, this is a natural choice in other settings such as statistical maximum likelihood estimation.

More generally, $\Mst$ can be a completely arbitrary low-rank matrix which we will typically think of as an \emph{approximate minimizer} of $\phi$;
the gradient $\nabla \phi(\Mst)$ is a suitable measure of error because we expect it to be small near the true minimizer(s) of $\phi$.

With this in mind, we have the following result generalizing \Cref{thm:matsens}; we prove this in \Cref{sec:proofs_landscape}.
\begin{theorem}
	\label{thm:noisy}
	Let $\phi$ be convex on $\Mspace$ (asymmetric case) or $\symms_d^+$ (symmetric case).
	For $r \geq \rstar \geq 1$,
	denote
	\[
		\mu \coloneqq \mu_{r + \rstar}(\phi), \quad L \coloneqq L_{r + \rstar}(\phi), \quad L_2 \coloneqq L_2^{(r + \rstar)}(\phi),
	\]
	and suppose
	\begin{equation}
		\label{eq:mu_cond}
		\mu > \frac{L_2}{2 \sqrt{r/\rstar} + L_2/L}.
	\end{equation}
	Set
	\begin{equation}
		\label{eq:mueff}
		\mueff \coloneqq \frac{1}{2} \parens*{ \sqrt{(L + \mu)^2 + \frac{\rstar}{r} L_2^2} - \sqrt{\frac{\rstar}{r}} L_2 - (L - \mu) } > 0,
	\end{equation}
	and let $\lambda \geq 0$. Finally, suppose that
	\begin{itemize}
		\item (Asymmetric case) $\Mst$ is any matrix in $\Mspace$ with rank at most $\rstar$,
		and $(U, V)$ is any second-order critical point of \eqref{eq:ncvx_asym}, and set $M = U V\transpose$; or
		\item (Symmetric case) $\Mst$ is any matrix in $\symms_d^+$ with rank at most $\rstar$,
		and $U$ is any second-order critical point of \eqref{eq:ncvx_sym}, and set $M = U U\transpose$.
	\end{itemize}
	Then
	\[
		\normF{M - \Mst} \leq \frac{ 6 \sqrt{\rstar} \lambda + \sqrt{r + \rstar} (\opnorm{\nabla \phi(\Mst)} - \lambda)_+ }{\mueff}.
	\]
\end{theorem}
A few remarks are in order:
\begin{remark}
	\label{rmk:Lequal}
	First, we can always simplify \eqref{eq:mu_cond} and \eqref{eq:mueff} by noting that $L_2 \leq L$ and replacing $L_2$ everywhere by $L$;
	we then obtain the simplfied condition
	\[
		\mu > \frac{L}{1 + 2 \sqrt{r/\rstar}}
	\]
	(note that this is equivalent to \eqref{eq:kappacrit} in \Cref{ex:overpbad}),
	and
	\[
		\mueff \geq \frac{1}{2} \parens*{ \sqrt{(L + \mu)^2 + \frac{\rstar}{r} L^2} - \sqrt{\frac{\rstar}{r}} L - (L - \mu) }.
	\]
\end{remark}

\begin{remark}
	\label{rmk:thmrecov_rip}
	How do we recover \Cref{thm:matsens} in the matrix sensing setting of \Cref{sec:intro}?
	As discussed above, under the model \eqref{eq:ms_model} with quadratic $\phi$ of the form \eqref{eq:quadraticphi},
	we have $\opnorm{\nabla \phi(\Mst)} = \opnorm{\scrA^*(\xi)}$.
	
	Furthermore, under \eqref{eq:rip} with $\genrk \geq r + \rstar$, we have $\mu \geq 1 - \deltak$ and $L_2 \leq L \leq 1 + \deltak$.
	Using the simplifications of \Cref{rmk:Lequal},
	the condition \eqref{eq:mu_cond} becomes
	\begin{gather*}
		1 - \deltak > \frac{1 + \deltak}{1 + 2 \sqrt{r/\rstar}}
		\qquad \Longleftrightarrow \qquad \deltak < \frac{1}{1 + \sqrt{\rstar/r}},
	\end{gather*}
	which is precisely \eqref{eq:deltacrit},
	and
	\begin{align*}
		\mueff &\geq \frac{1}{2} \parens*{ \sqrt{4 + \frac{\rstar}{r} (1 + \deltak)^2} - \sqrt{\frac{\rstar}{r}} (1 + \deltak) - 2 \deltak } \\
		&= \sqrt{1 + \frac{\rstar}{r} \parens*{\frac{1 + \deltak}{2}}^2} - \sqrt{\frac{\rstar}{r}} \frac{1 + \deltak}{2} - \deltak.
	\end{align*}
	Noting that
	\[
		\delta \mapsto \sqrt{1 + \frac{\rstar}{r} \parens*{\frac{1 + \delta}{2}}^2} - \sqrt{\frac{\rstar}{r}} \frac{1 + \delta}{2}
	\]
	is decreasing in $\delta$ and has value equal to $\deltac$ for $\delta = \deltac$,
	we can, for $\deltak < \deltac$, simplify the lower bound on $\mueff$ to
	\[
		\mueff \geq \deltac - \deltak.
	\]
	Hence we obtain \Cref{thm:matsens}.
\end{remark}

\begin{remark}
	\label{rmk:Llarge}
	In cases where $L \gg L_2$, we can take $L \to \infty$ in \Cref{thm:noisy} and obtain the following simplified versions of \eqref{eq:mu_cond} and \eqref{eq:mueff}:
	\begin{gather*}
		\mu > \frac{L_2}{2 \sqrt{r/\rstar}}, \\
		\mueff \geq \mu - \frac{L_2}{2 \sqrt{r/\rstar}}.
	\end{gather*}
	This result is, in fact, easier to prove than \Cref{thm:noisy};
	much of the complexity of our proof comes from showing the benefits from finite $L$.
	We will comment on this in the proof in \Cref{sec:proofs_landscape}.
\end{remark}

\subsection{Global optimality}
\label{sec:res_global}
The previous result \Cref{thm:noisy} says, intuitively,
that if $\phi$ has a low-rank approximate minimizer, then we can approximately recovery that minimizer via second-order critical points of the nonconvex nuclear-norm regularized problem \eqref{eq:ncvx_asym} or \eqref{eq:ncvx_sym}.
However, if, in addition, the regularized convex problem (\eqref{eq:cvx_asym} or \eqref{eq:cvx_sym}) has a \emph{low-rank global optimum},
we can say more:
second-order critical points of the nonconvex problem recover that optimum.

In the asymmetric case, $\Mopt$ is a global optimum of \eqref{eq:cvx_asym} if and only if it satisfies the subgradient condition
\begin{equation}
	\label{eq:global_asym}
	\nabla \phi(\Mopt) \in -\lambda \partial \nucnorm{\Mopt},
\end{equation}
where $\partial \nucnorm{M}$ is the subgradient of the nuclear norm at $M \in \Mspace$.

In the symmetric case, a matrix $\Mopt \succeq 0$ is globally optimal for \eqref{eq:cvx_sym} if and only if
\begin{align}
	\label{eq:global_sym_comp}
	(\nabla \phi(\Mopt) + \lambda I_d) \Mopt &= 0, \qquad \text{and} \\
	\label{eq:global_sym_dualfeas}
	\nabla \phi(\Mopt) + \lambda I_d &\succeq 0.
\end{align}
The first condition is a complementary slackness condition,
while the second corresponds to dual feasibility.

The following result says that, if there is a low-rank global optimum, we can recover it exactly under the same conditions as \Cref{thm:noisy}:
\begin{theorem}
	\label{thm:global_main}
	Under the conditions of \Cref{thm:noisy},
	suppose, in addition, that
	\begin{itemize}
		\item (Asymmetric case) $\Mst = \Mopt$ is a global optimum of \eqref{eq:cvx_asym} in the sense of \eqref{eq:global_asym}; or
		
		\item (Symmetric case) $\Mst = \Mopt$ is a global optimum of \eqref{eq:cvx_sym} in the sense of \eqref{eq:global_sym_comp} and \eqref{eq:global_sym_dualfeas}.
	\end{itemize}
	Then $M = \Mst = \Mopt$.
\end{theorem}
In particular, we are assuming that there is a global optimum $\Mopt$ of the convex problem with $\rank(\Mopt) \leq \rstar$.
See the remarks after \Cref{thm:noisy} for further discussion of the assumptions.
In the asymmetric case with $\lambda > 0$ and $L = L_2$ (see \Cref{rmk:Lequal}),
this recovers \cite[Cor.~1.4]{Ouyang2025}.

\Cref{thm:noisy,thm:global_main} are entirely distinct results\footnote{Except in the case $\nabla \phi(\Mst) = 0$ and $\lambda = 0$, when clearly $\Mst$ is a global optimum of the convex problem, and the error bound of \Cref{thm:noisy} is zero.} even if the proofs share many common elements (see \Cref{sec:proofs_landscape}).
For example, in the matrix sensing setting of \Cref{sec:intro}, there is no reason to expect, in general, that the ``ground truth'' matrix $\Mst$ of the model \eqref{eq:ms_model} will also be a global optimum of the matrix LASSO problem \eqref{eq:lasso}
(although this will indeed be the case in our carefully-constructed counterexamples in \Cref{sec:res_counter} below).

One might then reasonably ask: given \Cref{thm:global_main},
why do we need the ``recovery'' result \Cref{thm:noisy} (or its specialization, \Cref{thm:matsens}) when we already know (see \Cref{thm:lasso} and \Cref{subsec:relcvx}) that global optima of the convex formulations are good statistical estimators?
There are two reasons for this:
\begin{itemize}
	\item First, results for the convex matrix LASSO such as \Cref{thm:lasso} require $\lambda \gtrsim \opnorm{\scrA^*(\xi)} = \opnorm{\nabla \phi(\Mst)}$.
	\Cref{ex:error} (if we were to take $r$ large) makes clear that this is necessary to avoid an error that grows with the problem dimension.
	\Cref{thm:matsens,thm:noisy} do not require this and hence give a meaningful error bound for a broader range of problem parameters.
	
	\item Second, \Cref{thm:global_main} requires the \emph{a priori} existence of a low-rank solution $\Mopt$ to \eqref{eq:lasso}.
	Although this is what one intuitively expects when using nuclear norm regularization,
	there are few theoretical guarantees of this in the literature.
	Current results only hold under significantly stronger conditions than \Cref{thm:matsens,thm:noisy};
	see the recent works \cite{McRae2026,Agterberg2025} for some such results and further references. 
\end{itemize}
\subsection{Counterexamples}
\label{sec:res_counter}
In this section, we give two classes of counterexamples (with quadratic $\phi$ of the form \eqref{eq:quadraticphi}) showing that the condition \eqref{eq:mu_cond} in \Cref{thm:noisy,thm:global_main} is sharp. We prove these in \Cref{sec:proof_counter} as corollaries of a more general result, \Cref{lem:spur_gen}.

The first result shows that, for given $r$ and $\rstar$, the condition \eqref{eq:mu_cond} for \Cref{thm:noisy,thm:global_main} (and, equivalently, the condition \eqref{eq:deltacrit} for \Cref{thm:matsens}) is the tightest possible (assuming $L = L_2$; see \Cref{rmk:Lequal}).
This is a generalization of \cite[Ex.~6.2]{Zhang2025c} and \cite[Ex.~4.1 and 4.4]{Zhang2025a} to the case $\lambda > 0$.
\newcommand{\commonpts}{%
	\item For any $\lambda \geq 0$, there exists $\xi \in \R^n$ with $\opnorm{\scrA^*(\xi)} = \lambda$ such that $\Mst$ is the unique global optimum of the convex problem \eqref{eq:cvx_asym} (resp.\ \eqref{eq:cvx_sym}) with quadratic $\phi$ of the form \eqref{eq:quadraticphi} and $b = \scrA(\Mst) + \xi$.
}%
\begin{theorem}
	\label{thm:optmucond}
	In the asymmetric case (respectively, in the symmetric case),
	for any integers $r \geq \rstar \geq 1$ and $d \geq r + \rstar$,
	and for any $\mu < 1$ satisfying
	\begin{equation}
		\label{eq:mu_counter}
		\mu \geq \frac{1}{1 + 2 \sqrt{r/\rstar}},
	\end{equation}
	there is $\Mst \in \Mspaced$ (resp.\ $\symms_d^+$) with rank $\rstar$ and $\sigma_1(\Mst) = \sigma_{\rstar}(\Mst) = 1$,
	an integer $n$, and a linear operator $\scrA \colon \Mspaced \to \R^n$ (resp.\ $\scrA \colon \symms_d \to \R^n$) such that the following hold:
	\begin{itemize}
		\item $L_{\genrk}(\scrA) = 1$ and $\mu_{\genrk}(\scrA) \geq \mu$ for all $\genrk \geq 1$.
		
		\commonpts
		
		\item (Asymmetric case) The nonconvex problem \eqref{eq:ncvx_asym} with the given $\phi$, $\lambda$, and $r$ has a second-order critical point $(U, V)$ with $U \neq 0$, $V \neq 0$ such that $U\transpose \Mst = 0$ and $\Mst V = 0$, or, respectively,
		
		\item (Symmetric case) The nonconvex problem \eqref{eq:ncvx_sym} with the given $\phi$, $\lambda$, and $r$ has a second-order critical point $U \neq 0$ with $\Mst U = 0$.

		\item Furthermore, if the inequality \eqref{eq:mu_counter} is strict,
		then $(U, V)$ (resp.\ $U$) is a local minimum.
	\end{itemize}
\end{theorem}
Note that this is a counterexample to both \Cref{thm:noisy,thm:global_main} simultaneously.
\Cref{thm:noisy} would predict $\norm{M - \Mst} \lesssim \sqrt{\rstar} \lambda$ because $\opnorm{\phi(\Mst)} = \lambda$,
while \Cref{thm:global_main} would predict $M = \Mst$ because $\Mst$ is globally optimal.
However the counterexample shows a spurious second-order critical point with error $\normF{M - \Mst} \geq \sqrt{\rstar}$.


The following result generalizes \Cref{ex:overpbad},
showing that the requirement for restricted strong convexity and smoothness at rank $\genrk = r + \rstar$ in \Cref{thm:noisy,thm:global_main} cannot be relaxed even if $r \gg \rstar$:
\begin{theorem}
	\label{thm:overpbad}
	In the asymmetric case (respectively, in the symmetric case),
	for any $\mu \in (0, 1)$ and any integers $\rsmall \geq \rstar \geq 1$ and
	\[
		\rbig \geq \frac{\rstar \mu + \rsmall}{1 - \mu},
	\]
	there are integers $d$ and $n$, a matrix $\Mst \in \Mspaced$ (resp.\ $\symms_d^+$) with rank $\rstar$ and $\sigma_1(\Mst) = \sigma_{\rstar}(\Mst) = 1$, and a linear operator $\scrA \colon \Mspaced \to \R^n$ (resp.\ $\scrA \colon \symms_d \to \R^n$) such that the following hold:
	\begin{itemize}
		\item $L_{\genrk}(\scrA) = 1$ and $\mu_{\genrk}(\scrA) > 0$ for all $\genrk \geq 1$, and $\mu_{\rsmall + \rstar}(\scrA) = \mu$.
		
		\commonpts
		
		\item (Asymmetric case) The nonconvex problem \eqref{eq:ncvx_asym} with $r = \rbig$ and the given $\phi$ and $\lambda$ has a local minimum $(U, V)$ with $U \neq 0$, $V \neq 0$ such that $U\transpose \Mst = 0$ and $\Mst V = 0$, or, respectively,
		
		\item (Symmetric case) The nonconvex problem \eqref{eq:ncvx_sym} with $r = \rbig$ and the given $\phi$ and $\lambda$ has a local minimum $U \neq 0$ with $\Mst U = 0$.
	\end{itemize}
\end{theorem}
In particular, even if $\rsmall$, $\rstar$, and $\mu$ are such that \eqref{eq:mu_cond} is satisfied (with $L_2 = L = 1$),
which ensures, by \Cref{thm:noisy,thm:global_main}, that there is no spurious local optimum with $r = \rsmall$,
there is a larger rank parameter $\rbig$ for which the nonconvex problem has a spurious local minimum with $r = \rbig$.

\section{Preliminary calculations}
Before proceeding to the proofs of our main results in \Cref{sec:res_gen},
we make some standard preliminary calculations.
\subsection{Gradients and Hessians}
In the asymmetric case,
the gradient of $f_\lambda(U, V)$ as given in \eqref{eq:ncvx_asym} is
\begin{equation}
\label{eq:grad_asym}
\nabla f_\lambda(U, V) = \begin{bmatrix*}
	\nabla \phi(U V\transpose) V + \lambda U \\
	\nabla \phi(U V\transpose)\transpose U + \lambda V
\end{bmatrix*}.
\end{equation}
where the first block corresponds to $U$ and the second to $V$.
The Hessian quadratic form is
\begin{equation}
\label{eq:hess_asym}
\begin{aligned}
	\nabla^2 f_\lambda(U, V)[(\Udt, \Vdt), (\Udt, \Vdt)]
	&= 2 \ip{\nabla \phi(U V\transpose)}{\Udt \Vdt\transpose} + \lambda (\normF{\Udt}^2 + \normF{\Vdt}^2) \\
	&\qquad + \nabla^2 \phi(U V\transpose)[U \Vdt\transpose + \Udt V\transpose, U \Vdt\transpose + \Udt V\transpose]
\end{aligned}
\end{equation}
for all $\Udt \in \Lspace$, $\Vdt \in \Rspace$.

In the symmetric case,
the gradient of $f_\lambda(U)$ as given in \eqref{eq:ncvx_sym} is
\begin{equation}
	\label{eq:grad_sym}
	\nabla f_\lambda(U) = 2(\nabla \phi(U U\transpose) U + \lambda U).
\end{equation}
The Hessian quadratic form is
\begin{equation}
	\label{eq:hess_sym}
	\nabla^2 f_\lambda(U)[\Udt, \Udt]
	= 2 \ip{\nabla \phi(U U\transpose) + \lambda I_d}{\Udt \Udt\transpose} + \nabla^2 \phi(U U\transpose)[U \Udt\transpose + \Udt U\transpose, U \Udt\transpose + \Udt U\transpose]
\end{equation}
for all $\Udt \in \Uspace$.

In the important case where $\phi$ is quadratic of the form \eqref{eq:quadraticphi},
the gradient and Hessian of $\phi$ are
\begin{align}
	\label{eq:quadphi_grad}
	\nabla \phi(M) &= \scrA^*(\scrA(M) - b), \qquad \text{and} \\
	\label{eq:quadphi_hess}
	\nabla^2 \phi(M)[\Mdt, \Mdt]
	&= \norm{\scrA(\Mdt)}^2.
\end{align}

\subsection{Key inequality for restricted strongly convex and smooth functions}
The following inequality for gradients of restricted strongly convex and smooth functions (in the sense of \Cref{sec:rscs}) will be critical.
This is conceptually similar to the ``bounded difference property'' proposed by \cite{Bi2021}.
\begin{lemma}
	\label{lem:ip_bd}
	Let $\phi$ on $\Mspace$ (resp.\ $\symms_d^+$) be convex and twice differentiable.
	Let $\genrk \geq 1$,
	and denote $\mu \coloneqq \mu_{\genrk}(\phi)$, $L \coloneqq L_{\genrk}(\phi)$.
	For all $M_1, M_2 \in \Mspace$ (resp.\ $\symms_d^+$) and all $\genmat \in \Mspace$ (resp.\ $\symms_d$) such that
	\[
		\rank(a M_1 + b M_2 + c \genmat) \leq \genrk \quad \text{for all} \quad a, b, c \in \R,
	\]
	we have
	\[
		\abs*{ \ip{\nabla \phi(M_2) - \nabla \phi(M_1)}{\genmat} - \frac{L + \mu}{2} \ip{M_2 - M_1}{\genmat} } \leq \frac{L - \mu}{2} \normF{M_2 - M_1} \normF{\genmat}.
	\]
\end{lemma}
\begin{proof}
We use a standard integration argument (see, e.g.,
\cite{Li2019}).
Denote $\errmat \coloneqq M_2 - M_1$.
If $\genmat$ or $\errmat$ is zero, the result is trivial, so assume from now on that they are nonzero.
Rescaling if necessary, we can furthermore assume that $\normF{\genmat} = \normF{\errmat}$.
Then
\begin{align*}
	&\negqquad \ip{\nabla \phi(M_2) - \nabla \phi(M_1)}{\genmat} \\
	&= \int_0^1 \nabla^2 \phi(M_1 + t \errmat)[\errmat, \genmat] \ dt \\
	&= \frac{1}{4} \int_0^1 \parens*{ \nabla^2 \phi(M_1 + t \errmat)[\errmat + \genmat, \errmat + \genmat] - \nabla^2 \phi(M_1 + t \errmat)[\errmat - \genmat, \errmat - \genmat] } \ dt \\
	&\leq \frac{1}{4} (L \normF{\errmat + \genmat}^2 - \mu \normF{\errmat - \genmat}^2 ) \\
	&= \frac{L + \mu}{2} \ip{\errmat}{\genmat} + \frac{L - \mu}{4} (\normF{\errmat}^2 + \normF{\genmat}^2) \\
	&= \frac{L + \mu}{2} \ip{\errmat}{\genmat} + \frac{L - \mu}{2} \normF{\errmat} \normF{\genmat}.
\end{align*}
The inequality holds by $(\genrk, \genrk, \mu, L)$-restricted strong convexity and smoothness and because, by assumption, each of the matrices $M_1 + t \errmat$ (for all $t \in [0, 1]$), $\errmat + \genmat$, and $\errmat - \genmat$ has rank at most $\genrk$.
Similarly,
\begin{align*}
	\ip{\nabla \phi(M_2) - \nabla \phi(M_1)}{\genmat}
	&\geq \frac{1}{4} (\mu \normF{\errmat + \genmat}^2 - L \normF{\errmat - \genmat}^2 ) \\
	&= \frac{L + \mu}{2} \ip{\errmat}{\genmat} - \frac{L - \mu}{2} \normF{\errmat} \normF{\genmat}.
\end{align*}
This completes the proof.
\end{proof}

\section{Positive landscape proofs}
\label{sec:proofs_landscape}
\subsection{Unified proof}
\label{sec:proof_landscape_brief}
In this section we give a proof of \Cref{thm:global_main,thm:noisy}.
We will begin with a unified analysis for the two results,
specializing later to the specific assumptions on $\Mst$.
We defer certain technical details to lemmas proved in later subsections.
Thus, for now, we allow $\Mst$ to be an arbitrary rank-$\rstar$ matrix.
With $M$ as in the theorem statements,
we will set $\errmat = M - \Mst$ as the ``error'' matrix.
If $\errmat = 0$, we are already done,
so, from now on, we assume $\errmat \neq 0$.

First, we define a number of useful quantities related to $M$ and $\Mst$;
we state them as for the asymmetric case, but the symmetric case is identical with $d_1 = d_2 = d$, $V = U$ and some other minor adaptations that we will note.

We set $\PU \coloneqq U U^\dagger$ to be the orthogonal projection matrix onto $\range(U)$,
where $U^\dagger$ is the Moore-Penrose pseudoinverse of $U$, and we set $\PUp \coloneqq I_{d_1} - \PU$ as the orthogonal projection matrix onto $\range(U)^\perp$.
Similarly,
we set $\PV \coloneqq V V^\dagger$ and $\PVp \coloneqq I_{d_2} - \PV$.
We denote the subspace
\[
	\T \coloneqq \{ A V\transpose + U B\transpose  : A \in \Lspace, B \in \Rspace \} \subset \Mspace
\]
(in the symmetric case, we further enforce $A = B$).
The orthogonal projections onto $\T$ and its orthogonal complement $\Tp$ in $\Mspace$ are given, respectively, by
\begin{align*}
	\PT(Z) &= \PU Z + \PUp Z \PV = Z \PV + \PU Z \PVp, \qquad \text{and} \\
	\PTp(Z) &= \PUp Z \PVp.
\end{align*}
We then denote
\begin{align*}
	\Mp &\coloneqq \PTp(\Mst), \\
	\HT &\coloneqq \PT(\errmat), \qquad \qquad \qquad \text{and} \\
	\HTp &\coloneqq \PTp(\errmat) = - \Mp.
\end{align*}
Clearly, $\errmat = \HT + \HTp$.
The previous works \cite{Zhang2025a,Zhang2025c} obtain optimal results by considering in several places a rescaled error matrix of the form $t_1 \HT + t_2 \HTp$ for some carefully-chosen $t_1 \geq t_2 \geq 0$.
We will do the same here.
The decomposition
\begin{equation}
\label{eq:H_decomp}
t_1 \HT + t_2 \HTp = t_1 M + (t_1 - t_2) \Mp - t_1 M_*
\end{equation}
will be quite useful.

The following key lemma unifies the nonconvex landscape analysis of the symmetric and asymmetric cases:
\begin{lemma}
\label{lem:landscape_basic}
For any $t_1 \geq t_2 \geq 0$,
\[
\ip{\nabla \phi(M)}{t_1 \HT + t_2 \HTp}
\leq \lambda[ t_1 ( \nucnorm{\Mst} - \nucnorm{M} ) - (t_1 - t_2) \nucnorm{\Mp}]
+ t_2 L_2 \sigma_r(M) \nucnorm{\Mp}.
\]
\end{lemma}
We prove this in \Cref{sec:proof_basic_landscape}.
The symmetric case is a straightforward extension of arguments in \cite{Zhang2025a,Zhang2025c};
the asymmetric case is more novel and requires considerable additional care.

Subtracting $\ip{\nabla \phi(\Mst)}{t_1 \HT + t_2 \HTp}$ from both sides of \Cref{lem:landscape_basic}'s inequality,
we have
\begin{equation}
	\label{eq:basicineq1}
	\begin{aligned}
		\ip{\nabla \phi(M) - \nabla \phi(\Mst)}{t_1 \HT + t_2 \HTp}
		&\leq \lambda[ t_1 ( \nucnorm{\Mst} - \nucnorm{M} ) - (t_1 - t_2) \nucnorm{\Mp}] \\
		&\qquad + t_2 L_2 \sigma_r(M) \nucnorm{\Mp} - \ip{\nabla \phi(\Mst)}{t_1 \HT + t_2 \HTp}.
	\end{aligned}
\end{equation}

To handle the last term,
we will use the following technical result (which is key to the role of $r$ and $\rstar$):
\begin{lemma}
	\label{lem:richard}
	Suppose $\errmat \neq 0$, and set
	\begin{equation*}
		\alpha \coloneqq \frac{\normF{\Mp}}{\normF{\errmat}}, \quad \text{and} \quad \beta \coloneqq \begin{cases}
			\frac{\sigma_r(M)}{\normF{\errmat}} \cdot \frac{\nucnorm{\Mp}}{\normF{\Mp}} & \text{ if } \Mp \neq 0, \\
			0 & \text{ if } \Mp = 0.
		\end{cases}
	\end{equation*}
	Then
	\begin{equation*}
		\alpha^2 + \frac{r}{\rstar} \beta^2 \leq 1 + (\beta - \alpha)_+^2.
	\end{equation*}
\end{lemma}
In the symmetric case, this is \cite[Lem.~3.8]{Zhang2025c}.
We briefly describe in \Cref{sec:proof_validineq} below how that result can be extended to the asymmetric case.
The case $\Mp = 0$ is trivial, but we include it to avoid the need to explicitly handle it later.

With $\alpha, \beta$ as defined in \Cref{lem:richard} (recall that we have assumed $\errmat \neq 0$),
we can write
\begin{equation}
	\label{eq:ab_sub}
	\sigma_r(M) \nucnorm{\Mp} = \alpha \beta \normF{\errmat}^2.
\end{equation}

To lower-bound the left-hand side of \eqref{eq:basicineq1},
\Cref{lem:ip_bd} (with $\genrk = r + \rstar$) gives
\begin{align*}
&\negqquad \ip{\nabla \phi(M) - \nabla \phi(\Mst)}{t_1 \HT + t_2 \HTp} \\
&\geq \frac{L + \mu}{2} \ip{\errmat}{t_1 \HT + t_2 \HTp} - \frac{L - \mu}{2} \normF{\errmat} \normF{t_1 \HT + t_2 \HTp} \\
&= \frac{L + \mu}{2} ( t_1 \normF{\HT}^2 + t_2 \normF{\HTp}^2 ) - \frac{L - \mu}{2} \normF{\errmat} \sqrt{t_1^2 \normF{\HT}^2 + t_2^2 \normF{\HTp}^2} \\
&= \parens*{ \frac{L + \mu}{2} ( t_1 (1 - \alpha^2) + t_2 \alpha^2 ) - \frac{L - \mu}{2} \sqrt{t_1^2 (1 - \alpha^2) + t_2^2 \alpha^2} } \normF{\errmat}^2.
\end{align*}
To simplify the above expressions somewhat,
we will, similarly to \cite{Zhang2025c,Zhang2025a}, choose $t_1, t_2$ subject to the constraints
\begin{equation}
	\label{eq:t_constr}
	t_1 \geq t_2 \geq 0 \quad \text{and} \quad (1 - \alpha^2) t_1^2 + \alpha^2 t_2^2 = 1,
\end{equation}
which ensures that $\normF{t_1 \HT + t_2 \HTp} = \normF{\errmat}$.

Then, plugging the previous inequality and \eqref{eq:ab_sub} into \eqref{eq:basicineq1} and rearranging,
we obtain the delightful inequality
\begin{equation}
\label{eq:basicineq2}
\begin{aligned}
	&\negqquad \parens*{ \frac{L + \mu}{2} ( t_1 (1 - \alpha^2) + t_2 \alpha^2 ) - \frac{L - \mu}{2} - t_2 L_2 \alpha \beta } \normF{\errmat}^2 \\
	&\leq \lambda[ t_1 ( \nucnorm{\Mst} - \nucnorm{M} ) - (t_1 - t_2) \nucnorm{\Mp}]
	- \ip{\nabla \phi(\Mst)}{t_1 \HT + t_2 \HTp}.
\end{aligned}
\end{equation}
By \Cref{lem:richard},
we must furthermore have $\alpha^2 + \frac{r}{\rstar} \beta^2 \leq 1 + (\beta - \alpha)_+^2$.

We are free to choose any $t_1$ and $t_2$ subject to \eqref{eq:t_constr}.
A simple choice would be $t_1 = t_2 = 1$.
One can easily show that $\alpha \beta \leq \frac{1}{2 \sqrt{r/\rstar}}$, so
the left-hand side of \eqref{eq:basicineq2} would then be lower bounded by $(\mu - \frac{L_2}{2 \sqrt{r/\rstar}}) \normF{\errmat}^2$.
Indeed, for large $L$, we cannot do much better; see \Cref{rmk:Llarge}. Nevertheless, for finite $L$, we can do better with a more careful choice of $t_1$ and $t_2$.

The following lemma gives a tight lower bound on the coefficient of $\normF{\errmat}^2$ in \eqref{eq:basicineq2}:
\begin{lemma}
\label{lem:opt_messy}
For any $\rho \geq 1$, $L \geq \mu \geq 0$ with $L > 0$, and $L_2 \geq 0$,
\begin{gather*}
	\min_{\substack{\alpha, \beta \geq 0 \\ \alpha^2 + \rho^2 \beta^2 \leq 1 + (\beta - \alpha)_+^2 }}~
	\max_{ \substack{t_1 \geq t_2 \geq 0 \\ (1 - \alpha^2) t_1^2 + \alpha^2 t_2^2 = 1} }~
	\frac{L + \mu}{2} [(1 - \alpha^2) t_1 + \alpha^2 t_2] - \frac{L - \mu}{2} - t_2 L_2 \alpha \beta \\
	= \frac{1}{2} \parens*{ \sqrt{(L + \mu)^2 + \rho^{-2} L_2^2} - \rho^{-1} L_2 - (L - \mu) }.
\end{gather*}
This is strictly increasing in $\mu$ and is zero when
\[
	\mu = \frac{L_2}{2 \rho + L_2/L}.
\]
Furthermore, the maximal value of $t_1$ attained in the inner optimization is
\[
t_1 = \sqrt{1 + \parens*{ \frac{\rho^{-1} L_2}{L + \mu} }^2 } + \frac{\rho^{-1} L_2}{L + \mu}.
\]
\end{lemma}
This is a more general version of a result in \cite{Zhang2025c,Zhang2025a}.
We provide a proof in \Cref{sec:mueff_proof}.
We of course take $\rho = \sqrt{\frac{r}{\rstar}}$,
in which case the value is precisely $\mueff$ from \eqref{eq:mueff},
and the threshold for $\mu$ from the lemma gives the condition \eqref{eq:mu_cond} to ensure $\mueff > 0$.
Applying the lemma to \eqref{eq:basicineq2},
we obtain
\begin{equation}
\label{eq:basicineq3}
\mueff \normF{\errmat}^2
\leq \lambda[ \tbr_1 ( \nucnorm{\Mst} - \nucnorm{M} ) - (\tbr_1 - \tbr_2) \nucnorm{\Mp}]
- \ip{\nabla \phi(\Mst)}{\tbr_1 \HT + \tbr_2 \HTp}
\end{equation}
for some $\tbr_1, \tbr_2$ satisfying
\begin{equation}
\label{eq:tbrs}
0 \leq \tbr_2 \leq \tbr_1 \leq \sqrt{1 + \parens*{ \frac{\rho^{-1} L_2}{L + \mu} }^2 } + \frac{\rho^{-1} L_2}{L + \mu} \leq 1 + \sqrt{2}.
\end{equation}
From now on, we consider \Cref{thm:global_main,thm:noisy} separately.
\subsubsection{Proof of approximate recovery result (Thm.~\ref{thm:noisy})}
The inequality \eqref{eq:basicineq3},
together with the norm duality inequality
\[
	\abs{\ip{\nabla \phi(\Mst)}{\tbr_1 \HT + \tbr_2 \HTp}}
	\leq \opnorm{\nabla \phi(\Mst)} \nucnorm{\tbr_1 \HT + \tbr_2 \HTp},
\]
gives
\begin{equation}
	\label{eq:noisy_basicineq1}
	\mueff \normF{\errmat}^2
	\leq \lambda[ \tbr_1 ( \nucnorm{\Mst} - \nucnorm{M} ) - (\tbr_1 - \tbr_2) \nucnorm{\Mp}]
	+ \opnorm{\nabla \phi(\Mst)} \nucnorm{\tbr_1 \HT + \tbr_2 \HTp}.
\end{equation}
Recalling \eqref{eq:H_decomp},
we can bound
\begin{equation}
\label{eq:nnbd_decomp}
\nucnorm{\tbr_1 \HT + \tbr_2 \HTp}
\leq \tbr_1 \nucnorm{\errmat} + (\tbr_1 - \tbr_2) \nucnorm{\Mp}.
\end{equation}
It is tempting to plug this directly into \eqref{eq:noisy_basicineq1},
but the perplexing term $(\tbr_1 - \tbr_2) \nucnorm{\Mp}$ will only get canceled if $\lambda \geq \opnorm{\nabla \phi(\Mst)}$.

To use \eqref{eq:nnbd_decomp} while allowing for $\opnorm{\nabla \phi(\Mst)} > \lambda$,
we partition $\opnorm{\nabla \phi(\Mst)}$ as
\newcommand{\nsbig}{\sigma_1}
\newcommand{\nssmall}{\sigma_2}
\begin{equation}
	\label{eq:noise_part}
	\opnorm{\nabla \phi(\Mst)}
	= \nsbig + \nssmall, \quad \text{where} \quad \nsbig \coloneqq (\opnorm{\nabla \phi(\Mst)} - \lambda)_+, \quad \nssmall \coloneqq \min\{ \opnorm{\nabla \phi(\Mst)}, \lambda \}.
\end{equation}
We will use \eqref{eq:nnbd_decomp} with $\nssmall \leq \lambda$.
For the other term $\nsbig$,
we use the bound
\begin{equation*}
	\nucnorm{\tbr_1 \HT + \tbr_2 \HTp}
	\leq \sqrt{r + \rstar} \normF{\tbr_1 \HT + \tbr_2 \HTp}
	= \sqrt{r + \rstar} \normF{\errmat}.
\end{equation*}
Plugging these into \eqref{eq:noisy_basicineq1},
we obtain
\begin{equation}
\label{eq:noisy_basicineq2}
\begin{aligned}
	\mueff \normF{\errmat}^2
	&\leq \lambda[ \tbr_1 ( \nucnorm{\Mst} - \nucnorm{M} ) - (\tbr_1 - \tbr_2) \nucnorm{\Mp}] \\
	&\qquad	+ \nsbig \sqrt{r + \rstar} \normF{\errmat} + \nssmall [\tbr_1 \nucnorm{\errmat} + (\tbr_1 - \tbr_2) \nucnorm{\Mp}] \\
	&\leq \tbr_1 \lambda [ \nucnorm{\Mst} - \nucnorm{M} + \nucnorm{\errmat} ] + \sqrt{r + \rstar} \nsbig \normF{\errmat}.
\end{aligned}
\end{equation}
There is nothing more we can do with the last term,
but it will be zero if $\lambda \geq \opnorm{\nabla \phi(\Mst)}$.
We now turn our attention to bounding  $\nucnorm{\Mst} - \nucnorm{M} + \nucnorm{\errmat}$.

We use a type of argument that is standard in the low-rank matrix recovery literature.
Similarly to the decomposition of $\errmat$ into $\HT$ and $\HTp$,
we now
decompose $\errmat$ in subspaces depending on the matrix $\Mst$.
Denote the singular value decomposition of $\Mst$ by $\Mst = \Ustbr \Sigma_* \Vstbr\transpose$,
where $\Ustbr \in \Lstspace, \Vstbr \in \Rstspace$ with $\Ustbr\transpose \Ustbr = \Vstbr\transpose \Vstbr = I_{\rstar}$,
and $\Sigma_* \in \R^{\rstar \times \rstar}$
is diagonal and strictly positive definite.

We set $\PUst = \Ustbr \Ustbr\transpose$, $\PUstp = I_{d_1} - \PUst$,
$\PVst = \Vstbr \Vstbr\transpose$, and $\PVstp = I_{d_2} - \PVst$,
and we denote the tangent space to $\Mst$ by
\[
\Tst = \{ A \Vstbr\transpose + \Ustbr B\transpose  : A \in \Lstspace, B \in \Rstspace \} \subset \Mspace.
\]
(In the symmetric case, we again have $\Ustbr = \Vstbr$ and $A = B$.)
The orthogonal projections onto $\Tst$ and $\Tstp$ are given by
\begin{align*}
\PTst(Z) &= \PUst Z + \PUstp Z \PVst = Z \PVst + \PUst Z \PVstp, \qquad \text{and} \\
\PTstp(Z) &= \PUstp Z \PVstp.
\end{align*}
We write $\HTst = \PTst(\errmat)$ and $\HTstp = \PTstp(\errmat) = \PTstp(M)$;
note that $\errmat = \HTst + \HTstp$.
By the fact that $\rank(\HTst) \leq 2 \rstar$,
we have
\begin{equation}
	\label{eq:nnH_ub}
	\nucnorm{\errmat} \leq \nucnorm{\HTst} + \nucnorm{\HTstp}
	\leq \sqrt{2 \rstar} \normF{\HTst} + \nucnorm{\HTstp}.
\end{equation}
Next, critically,
\[
\nucnorm{M} - \nucnorm{\Mst}
\geq \ip{\errmat}{G} \quad \text{for any} \quad G \in \partial\nucnorm{\Mst} = \{ \Ustbr \Vstbr\transpose + \Wp : \Wp \in \Tstp, \opnorm{\Wp} \leq 1 \}.
\]
In particular, we can choose $\Wp$ such that $\ip{\HTstp}{\Wp} = \nucnorm{\HTstp}$.
We then have, noting that $\normF{\Ustbr \Vstbr\transpose} = \sqrt{\rstar}$,
\begin{equation*}
\begin{aligned}
	\nucnorm{M} - \nucnorm{\Mst}
	&\geq 
	\ip{\HTst}{\Ustbr \Vstbr\transpose} + \ip{\HTstp}{\Wp} \\
	&\geq - \sqrt{\rstar} \normF{\HTst} + \nucnorm{\HTstp}.
\end{aligned}
\end{equation*}
Subtracting this from \eqref{eq:nnH_ub},
we obtain
\[
	\nucnorm{\Mst} - \nucnorm{M} + \nucnorm{\errmat}
	\leq (1 + \sqrt{2})\sqrt{\rstar} \normF{\HTst}
	\leq (1 + \sqrt{2})\sqrt{\rstar} \normF{\errmat}.
\]
Substituting this last inequality into \eqref{eq:noisy_basicineq2}
along with the bound on $\tbr_1$ from \eqref{eq:tbrs},
we obtain
\[
\mueff \normF{\errmat}^2 \leq [(1 + \sqrt{2})^2 \sqrt{\rstar} \lambda + \sqrt{r + \rstar} \nsbig] \normF{\errmat}.
\]
Recalling the value of $\nsbig$ from \eqref{eq:noise_part} and noting that $(1 + \sqrt{2})^2 \leq 6$,
we obtain \Cref{thm:noisy}.

\subsubsection{Proof of global optimality (Thm.~\ref{thm:global_main})}
We now assume that $\Mst$ is a global optimum of \eqref{eq:cvx_asym} or \eqref{eq:cvx_sym}.
We will show, in both the symmetric and asymmetric cases, that, for any $t_1 \geq t_2 \geq 0$,
\begin{equation}
\label{eq:gradMst_lb}
\ip{\nabla \phi(\Mst)}{t_1 \HT + t_2 \HTp}
\geq \lambda [ t_1 (\nucnorm{\Mst} - \nucnorm{M}) - (t_1 - t_2) \nucnorm{\Mp} ].
\end{equation}
Plugging this into \eqref{eq:basicineq3} (with $t_1 = \tbr_1, t_2 = \tbr_2$),
we obtain
\[
\mueff \normF{\errmat}^2 \leq 0,
\]
implying $\errmat = 0$ (in fact, this contradicts our assumption that $\errmat \neq 0$).

To prove \eqref{eq:gradMst_lb},
recall that, in the symmetric case, we assumed $\Mst$ is a global optimum of \eqref{eq:cvx_sym} in the sense of \eqref{eq:global_sym_comp} and \eqref{eq:global_sym_dualfeas}.
Note that
\begin{align*}
\ip{\nabla \phi(\Mst)}{t_1 \HT + t_2 \HTp}
&= \ip{\nabla \phi(\Mst) + \lambda I_d}{t_1 \HT + t_2 \HTp}
- \lambda \tr(t_1 \HT + t_2 \HTp).
\end{align*}
Recalling the decomposition \eqref{eq:H_decomp} and noting that $M, \Mst, \Mp \succeq 0$,
we have
\[
\tr(t_1 \HT + t_2 \HTp) = t_1 (\nucnorm{M} - \nucnorm{\Mst}) + (t_1 - t_2) \nucnorm{\Mp},
\]
and, by \eqref{eq:global_sym_comp} and \eqref{eq:global_sym_dualfeas},
\begin{align*}
&\negqquad \ip{\nabla \phi(\Mst) + \lambda I_d}{t_1 \HT + t_2 \HTp} \\
&= \ip{\underbrace{\nabla \phi(\Mst) + \lambda I_d}_{\succeq 0}}{\underbrace{t_1 M + (t_1 - t_2) \Mp }_{\succeq 0} } - t_1 \underbrace{\ip{\nabla \phi(\Mst) + \lambda I_d}{\Mst}}_{= 0} \\
&\geq 0.
\end{align*}
Thus we obtain \eqref{eq:gradMst_lb}.

In the asymmetric case, we assumed that $\Mst$ is a global optimum of \eqref{eq:cvx_asym} in the sense of \eqref{eq:global_asym}.
Again recalling the decomposition \eqref{eq:H_decomp},
this implies that
\begin{align*}
	\ip{\nabla \phi(\Mst)}{t_1 \HT + t_2 \HTp}
	&= t_1 [ \ip{-\nabla \phi(\Mst)}{ \Mst } - \ip{-\nabla \phi(\Mst)}{ M }] - (t_1 - t_2) \ip{-\nabla \phi(\Mst)}{\Mp} \\
	&\geq \lambda [ t_1 (\nucnorm{\Mst} - \nucnorm{M}) - (t_1 - t_2) \nucnorm{\Mp} ],
\end{align*}
where the inequality follows from the fact (by \eqref{eq:global_asym}) that $-\nabla \phi(\Mst) \in \lambda \partial \nucnorm{\Mst}$.
Thus we once again have \eqref{eq:gradMst_lb}.

This completes the proof of \Cref{thm:global_main}.

\subsection{Basic landscape lemma proof}
\label{sec:proof_basic_landscape}
In this section,
we prove \Cref{lem:landscape_basic}.
We prove it separately in the symmetric and asymmetric cases.
The symmetric case is a straightforward extension of arguments in \cite{Zhang2025c,Zhang2025a};
the asymmetric case requires considerable additional work.

\subsubsection{Basic landscape proof: symmetric case}
In this case, $M = U U\transpose$.
Recalling \eqref{eq:grad_sym} and \eqref{eq:hess_sym},
second-order criticality of $U$ for \eqref{eq:ncvx_sym} is equivalent to
\begin{align}
(\nabla \phi(M) + \lambda I_d) U = 0 \label{eq:socp_cond_1}
\end{align}
and, for all $\Udt \in \Uspace$,
\begin{align}
\ip{\nabla \phi(M) + \lambda I_d}{\Udt \Udt\transpose} + \frac{1}{2} \nabla^2 \phi(M)[U \Udt\transpose + \Udt U\transpose, U \Udt\transpose + \Udt U\transpose] \geq 0. \label{eq:socp_cond_2}
\end{align}
Choose $\Udt = v w\transpose$ for some $v \in \R^d$ such that $U\transpose v = 0$ and $w \in \R^r$. We will specify these later.
Then, by restricted smoothness and the fact that $U \Udt\transpose + \Udt U\transpose$ has rank at most $2$,
\begin{align*}
&\negqquad \frac{1}{2} \nabla^2 \phi(M)[U \Udt\transpose + \Udt U\transpose, U \Udt\transpose + \Udt U\transpose] \\
&= \frac{1}{2} \nabla^2 \phi(M)[U w v\transpose + v (Uw)\transpose, U w v\transpose + v (Uw)\transpose] \\
&\leq \frac{L_2}{2} \normF{ U w v\transpose + v (Uw)\transpose }^2 \\
&= L_2 \norm{Uw}^2 \norm{v}^2.
\end{align*}
The last equality used the fact that $U\transpose v = 0$.
We then choose $w$ to be the unit-norm vector that minimizes $\norm{Uw}$, for which $\norm{U w}^2 = \sigma_r^2(U) = \sigma_r(M)$.
Then the previous inequality and \eqref{eq:socp_cond_2} imply
\[
	0 \leq \ip{\nabla \phi(M) + \lambda I_d}{v v\transpose} + L_2 \sigma_r(M) \norm{v}^2.
\]
Write $\Mp = \Up \Up\transpose$ for some $\Up \in \R^{d \times \rstar}$.
Choosing $v$ to be each of the $\rstar$ columns of $\Up$ (thus satisfying $U\transpose v = 0$) and adding the resulting $\rstar$ inequalities, we obtain
\begin{equation}
	\label{eq:ineq_N}
	0 \leq \ip{\nabla \phi(M) + \lambda I_d}{\Up \Up\transpose} + L_2 \sigma_r(M) \normF{\Up}^2
	= \ip{\nabla \phi(M) + \lambda I_d}{\Mp} + L_2 \sigma_r(M) \nucnorm{\Mp}.
\end{equation}
Next, note that \eqref{eq:socp_cond_1} implies $\ip{\nabla \phi(M) + \lambda I_d}{\HT} = 0$.
Combining this with \eqref{eq:ineq_N} and recalling that $\HTp = -\Mp$ give
\begin{equation}
	\label{eq:ub_t}
	\ip{\nabla \phi(M) + \lambda I_d}{t_1 \HT + t_2 \HTp}
	\leq t_2 L_2 \sigma_r(M) \nucnorm{\Mp}.
\end{equation}
Recalling the decomposition \eqref{eq:H_decomp} and the fact that $M, \Mst, \Mp \succeq 0$, we have
\[
	\tr( t_1 \HT + t_2 \HTp )
	= t_1 (\nucnorm{M} - \nucnorm{\Mst}) + (t_1 - t_2) \nucnorm{\Mp}.
\]
Plugging this into \eqref{eq:ub_t} and rearranging completes the proof of \Cref{lem:landscape_basic} in the symmetric case.
\subsubsection{Basic landscape proof: asymmetric case}
In this case, $M = U V\transpose$.
Recalling \eqref{eq:grad_asym} and \eqref{eq:hess_asym},
second-order criticality of $(U, V)$ for \eqref{eq:ncvx_asym} is equivalent to
\begin{equation}
	\label{eq:asym_zerograd}
	\begin{aligned}
		\nabla \phi(M) V + \lambda U &= 0, \quad \text{and} \\
		\nabla \phi(M)\transpose U + \lambda V &= 0,
	\end{aligned}
\end{equation}
and, for all $\Udt \in \Lspace$ and $\Vdt \in \Rspace$,
\begin{equation}
	\label{eq:asym_hesspos}
	\ip{\nabla \phi(M)}{\Udt \Vdt\transpose} + \frac{\lambda}{2} (\normF{\Udt}^2 + \normF{\Vdt}^2) + \frac{1}{2} \nabla^2 \phi(M)[U \Vdt\transpose + \Udt V\transpose, U \Vdt\transpose + \Udt V\transpose] \geq 0.
\end{equation}

From these conditions, we can derive an inequality similar to \eqref{eq:ineq_N} from the symmetric case:
\begin{lemma}
\label{lem:asym_Mperp}
If $(U, V)$ is a second-order critical point of \eqref{eq:ncvx_asym}, then
\begin{equation*}
	0 \leq \ip{\nabla \phi(M)}{\Mp} + \lambda \nucnorm{\Mp} + L_2 \sigma_r(M) \nucnorm{\Mp}.
\end{equation*}
\end{lemma}
We defer the proof to later in this section,
as it is considerably more involved than in the symmetric case.
The rest of the proof is conceptually similar to that for the symmetric case,
though additional care is needed with the nuclear norm terms.

Let $Q_\perp \in \partial \nucnorm{\Mp}$ such that $Q_\perp \transpose U = 0$ and $Q_\perp V = 0$
(we can do this because $\Mp\transpose U = 0$ and $\Mp V = 0$).
As $\HTp = - \Mp$, \Cref{lem:asym_Mperp} implies that
\begin{equation}
\label{eq:bd_term2}
\begin{aligned}
	\ip{\nabla \phi(M)}{t_2 \HTp}
	&\leq -\lambda \ip{\Qp}{t_2 \HTp} + t_2 L_2 \sigma_r(M) \nucnorm{\Mp} \\
	&= -\lambda \ip{\Qp}{t_1 \HT + t_2 \HTp} + t_2 L_2 \sigma_r(M) \nucnorm{\Mp}.
\end{aligned}
\end{equation}
Next, let $Q_M \in \partial \nucnorm{M}$ be the matrix sign of $M$,
that is, if $M = \Ubr \Sigma \Vbr\transpose$ is the singular value decomposition of $M$ with $\Ubr \in \R^{d_1 \times \rank(M)}$ and $\Vbr \in \R^{d_2 \times \rank(M)}$ having orthonormal columns,
we set $Q_M = \Ubr \Vbr\transpose$.
The zero-gradient conditions \eqref{eq:asym_zerograd} imply
\[
	\PT( \nabla \phi(M) )= - \lambda Q_M.
\]
This is trivial if $\lambda = 0$; if $\lambda > 0$, it follows from some simple manipulation of \eqref{eq:asym_zerograd} (see, e.g., \cite[Sec.~5.2]{McRae2026}).
In particular, this implies that
\begin{align*}
\ip{\nabla \phi(M)}{t_1 \HT}
&= - \lambda \ip{Q_M}{t_1 \HT} \\
&= - \lambda \ip{Q_M}{t_1 \HT + t_2 \HTp}.
\end{align*}
Together with \eqref{eq:bd_term2},
this implies
\begin{equation}
\label{eq:bd_term_comb}
\ip{\nabla \phi(M)}{t_1 \HT + t_2 \HTp}
\leq - \lambda \ip{Q_M + \Qp}{t_1 \HT + t_2 \HTp} + t_2 L_2 \sigma_r(M) \nucnorm{\Mp}.
\end{equation}

Recalling the decomposition \eqref{eq:H_decomp},
as $Q_M + Q_\perp \in \partial \nucnorm{t_1 M + (t_1 - t_2) \Mp}$,
we have
\begin{align*}
\ip{Q_M + \Qp}{t_1 \HT + t_2 \HTp}
&= \ip{Q_M + \Qp}{t_1 M + (t_1 - t_2) \Mp - t_1 \Mst} \\
&\geq \nucnorm{t_1 M + (t_1 - t_2) \Mp} - t_1 \nucnorm{\Mst} \\
&= (t_1 - t_2) \nucnorm{\Mp} + t_1 ( \nucnorm{M} - \nucnorm{\Mst} ).
\end{align*}
Plugging this into \eqref{eq:bd_term_comb} gives \Cref{lem:landscape_basic} in the asymmetric case.

\newcommand{\udto}{u}
\newcommand{\udti}{w}
\newcommand{\vdto}{v}
\newcommand{\vdti}{x}
We finish with a proof of the technical lemma above.
In the case $\lambda = 0$, the proof resembles in some ways that of \cite[Thm.~2.3(a)]{Ha2020}.
\begin{proof}[Proof of \Cref{lem:asym_Mperp}]
Similarly to the symmetric case,
we will use \eqref{eq:asym_hesspos} with careful choices of $\Udt \in \Lspace$, $\Vdt \in \Rspace$.
However, making the correct choices requires a bit more work than in the symmetric case.
We will again choose rank-1 directions:
we take $\Udt = \udto \udti\transpose$, $\Vdt = \vdto \vdti\transpose$ for some $\udto \in \R^{d_1}$, $\vdto \in \R^{d_2}$, and $\udti, \vdti \in \R^{r}$.
With this choice, \eqref{eq:asym_hesspos} gives
\begin{align*}
	0 &\leq \ip{\nabla \phi(M)}{\udto \vdto\transpose}\ip{\udti}{\vdti} + \frac{\lambda}{2} (\norm{\udto}^2 \norm{\udti}^2 + \norm{\vdto}^2 \norm{\vdti}^2) \\
	&\qquad +  \frac{1}{2} \nabla^2 \phi(M)[U \vdti \vdto\transpose + \udto (V \udti)\transpose, U \vdti \vdto\transpose + \udto (V \udti)\transpose].
\end{align*}
If we furthermore choose $\udto, \vdto$ such that $U\transpose \udto = V\transpose \vdto = 0$,
restricted smoothness implies
\begin{align*}
	\frac{1}{2} \nabla^2 \phi(M)[U \vdti \vdto\transpose + \udto (V \udti)\transpose, U \vdti \vdto\transpose + \udto (V \udti)\transpose]
	&\leq \frac{L_2}{2} \normF{U \vdti \vdto\transpose + \udto (V \udti)\transpose}^2 \\
	&= \frac{L_2}{2} ( \norm{\udto}^2 \norm{V \udti}^2 + \norm{\vdto}^2 \norm{U \vdti}^2).
\end{align*}
Plugging this into the previous inequality gives
\begin{equation*}
	\begin{aligned}
		0 &\leq \ip{\nabla \phi(M)}{\udto \vdto\transpose}\ip{\udti}{\vdti} + \frac{\lambda}{2} (\norm{\udto}^2 \norm{\udti}^2 + \norm{\vdto}^2 \norm{\vdti}^2) \\
		&\qquad+ \frac{L_2}{2} ( \norm{\udto}^2 \norm{V \udti}^2 + \norm{\vdto}^2 \norm{U \vdti}^2).
	\end{aligned}
\end{equation*}
To choose $\udto$ and $\vdto$,
write $\Mp = \Up \Vp\transpose$ for some $\Up \in \R^{d_1 \times \rstar}$ and $\Vp \in \R^{d_2 \times \rstar}$ satisfying $\Up\transpose \Up = \Vp\transpose \Vp$.
If $(\udto_1, \dots \udto_{\rstar})$ are the columns of $\Up$,
and $(\vdto_1, \dots \vdto_{\rstar})$ are the columns of $\Vp$,
then
\[
	\sum_{j = 1}^{\rstar} \udto_j \vdto_j\transpose = \Mp, \qquad \text{and} \qquad \sum_{j = 1}^{\rstar} \norm{\udto_j}^2 = \sum_{j = 1}^{\rstar} \norm{\vdto_j}^2 = \nucnorm{\Mp}.
\]
Note that these indeed satisfy $U\transpose \udto_j = V\transpose \vdto_j = 0$ for each $j$.
Taking $\udto = \udto_j, \vdto = \vdto_j$ in the previous inequality and summing over $j$,
we obtain
\begin{equation}
	\label{eq:asym_hess_prelim}
	0 \leq \ip{\udti}{\vdti} \ip{\nabla \phi(M)}{\Mp} + \parens*{ \lambda \frac{\norm{\udti}^2 + \norm{\vdti}^2}{2} + L_2 \frac{\norm{V \udti}^2 + \norm{U \vdti}^2}{2} } \nucnorm{\Mp}.
\end{equation}
\newcommand{\Ubal}{U_{\mathrm{bal}}}
\newcommand{\Vbal}{V_{\mathrm{bal}}}
To choose $\udti$ and $\vdti$ in \eqref{eq:asym_hess_prelim}, we consider several cases (this is similar to the proofs in \cite{Ha2020,Kim2025a}):
\begin{itemize}
	\item \textbf{Case 1:} $\lambda > 0$.
	In this case, the first-order conditions \eqref{eq:asym_zerograd} imply (see, e.g., \cite[Prop.~4.3]{Li2019}) that $(U, V)$ is balanced, that is, $U\transpose U = V\transpose V$.
	There is then a unit-norm $h \in \R^r$ such that $\norm{Uh}^2 = \norm{Vh}^2 = \sigma_r(M)$.
	We then take $\udti = \vdti = h$ in \eqref{eq:asym_hess_prelim} to obtain the result.
	\item \textbf{Case 2:} $\lambda = 0$, $\rank(M) = r$. With $\lambda = 0$, we can no longer assume $(U, V)$ is balanced.
	Let $M = \Ubal \Vbal\transpose$ be a balanced factorization of $M$ with $\Ubal\transpose \Ubal = \Vbal\transpose \Vbal$.
	There is then a unit-norm $h \in \R^r$ with $\norm{\Ubal h}^2 = \norm{\Vbal h}^2 = \sigma_r(M)$.
	The fact that $\rank(M) = r$ ensures $U$, $\Ubal$, $V$, and $\Vbal$ are all full-rank, $\range(U) = \range(\Ubal)$, and $\range(V) = \range(\Vbal)$.
	Therefore, $U U^\dagger \Ubal = \Ubal$, $V V^\dagger \Vbal = \Vbal$,
	and
	\[
	U V\transpose = M = \Ubal \Vbal\transpose \quad \Longrightarrow \quad U^\dagger \Ubal \Vbal\transpose (V^\dagger)\transpose = I_r
	\quad \Longrightarrow \quad (V^\dagger \Vbal)\transpose U^\dagger \Ubal = I_r.
	\]
	Then, setting $\udti = V^\dagger \Vbal h$ and $\vdti = U^\dagger \Ubal h$,
	we have $\norm{V\udti}^2 = \norm{U \vdti}^2 = \sigma_r(M)$,
	and
	\[
		\ip{\udti}{\vdti} = \ip{V^\dagger \Vbal h}{ U^\dagger \Ubal h } = \ip{h}{ (V^\dagger \Vbal)\transpose U^\dagger \Ubal h}
		= \ip{h}{h}
		= 1.
	\]
	We then obtain, from \eqref{eq:asym_hess_prelim} (with $\lambda = 0$),
	\[
	0 \leq \ip{\nabla \phi(M)}{\Mp} + L_2 \sigma_r(M) \nucnorm{\Mp},
	\]
	which is precisely the claimed result in the case $\lambda = 0$.
	\item \textbf{Case 3:} $\lambda = 0$, $\rank(M) < r$.
	In this case, $\sigma_r(M) = 0$, and at least one of $U$ and $V$ is rank-deficient.
	Assume, without loss of generality (otherwise transpose), that $\rank(U) < r$.
	Then, there exists a unit-norm $h \in \R^r$ such that $U h = 0$.
	Choose $\udti = s h$, $\vdti = s^{-1} h$ for $s > 0$;
	the inequality \eqref{eq:asym_hess_prelim} (with $\lambda = 0$) then becomes
	\begin{align*}
		0 &\leq \ip{\nabla \phi(M)}{\Mp} + L_2 \frac{s^2 \norm{V h}^2 + s^{-2} \norm{U h}^2}{2} \nucnorm{\Mp} \\
		&= \ip{\nabla \phi(M)}{\Mp} + s^2 L_2 \frac{\norm{V h}^2}{2} \nucnorm{\Mp}.
	\end{align*}
	Taking $s \to 0$ gives $0 \leq \ip{\nabla \phi(M)}{\Mp}$,
	which is precisely the result in the case $\lambda = 0$ and $\sigma_r(M) = 0$.
\end{itemize}
This completes the proof.
\end{proof}

\subsection{The key singular value inequality}
\label{sec:proof_validineq}
In this section, we briefly describe how the proof of \cite[Lem.~3.8]{Zhang2025c} can be extended to the general asymmetric case to give our \Cref{lem:richard}.

If $\Mp$ or $\sigma_r(M)$ is zero, then $\beta = 0$, and the result holds because $\alpha = \frac{\normF{\HTp}}{\normF{\errmat}} \leq 1$.
Thus, from now on, assume $\Mp \neq 0$ and $\sigma_r(M) > 0$.
In particular, with $M = U V\transpose$, we have $\rank(U) = \rank(V) = \rank(M) = r$.

\newcommand{\Usmall}{\Utl}
\newcommand{\Vsmall}{\Vtl}
\newcommand{\Ustt}{\Utl_{*}}
\newcommand{\Vstt}{\Vtl_{*}}
\newcommand{\Ustp}{U_{\perp}}
\newcommand{\Vstp}{V_{\perp}}
We factor $\Mst$ as $\Mst = \Ust \Vst\transpose$ for some $\Ust \in \Lstspace$, $\Vst \in \Rstspace$.
By an orthogonal change of basis (to which $\alpha$ and $\beta$ are invariant),
we can write
\[
	U = \begin{bmatrix*} \Usmall \\ 0_{(d_1 - r) \times r} \end{bmatrix*}, \quad V = \begin{bmatrix*} \Vsmall \\ 0_{(d_2 - r) \times r} \end{bmatrix*}, \quad \Ust = \begin{bmatrix*} \Ustt \\ \Ustp \end{bmatrix*}, \quad \Vst = \begin{bmatrix*} \Vstt \\ \Vstp \end{bmatrix*},
\]
where $\Usmall, \Vsmall \in \R^{r \times r}$, $\Ustt, \Vstt \in \R^{r \times \rstar}$, $\Ustp \in \R^{(d_1 - r) \times \rstar}$, and $\Vstp \in \R^{(d_2 - r) \times \rstar}$
As $\Usmall$ and $\Vsmall$ have full rank, we must have
\[
	\Mp = \begin{bmatrix*}
		0 & 0 \\
		0 & \Ustp \Vstp\transpose
	\end{bmatrix*},
\]
so $\Mp$ and $\Ustp \Vstp\transpose$ have the same nonzero singular values (this is why it was important to assume $\sigma_r(M) > 0$).

We can then decompose
\begin{align*}
	\normF{\errmat}^2
	&= \normF{U V\transpose - \Ust \Vst\transpose}^2 \\
	&= \normF{\Usmall \Vsmall\transpose - \Ustt \Vstt\transpose}^2 + \normF{\Ustp \Vstt\transpose}^2 + \normF{\Ustt \Vstp\transpose}^2 + \normF{\Ustp \Vstp\transpose}^2 \\
	&= \normF{\Usmall \Vsmall\transpose - \Ustt \Vstt\transpose}^2 + \ip{\Ustp\transpose \Ustp}{\Vstt\transpose \Vstt} + \ip{\Vstp\transpose \Vstp}{\Ustt\transpose \Ustt} + \normF{\Mp}^2.
\end{align*}
Similarly to \cite[Sec.~4]{Zhang2025c}, we write the vectors
\begin{align*}
	s &\coloneqq (\sigma_1(M), \dots, \sigma_r(M)) \in \R^r, \\
	s' &\coloneqq (\sigma_1(M), \dots, \sigma_{\rstar}(M)) \in \R^{\rstar}, \\
	d &\coloneqq (\sigma_{\rstar}(\Mp), \dots, \sigma_1(\Mp)) \in \R^{\rstar}.
\end{align*}
In particular, note that $s_1 \geq \cdots \geq s_r > 0$, while $0 \leq d_1 \leq \cdots \leq d_{\rstar}$.
Critically, we can assume, without loss of generality (if necessary replacing $(\Ust, \Vst)$ by $(\Ust R, \Vst (R^{-1})\transpose)$ for some $R \in \R^{\rstar \times \rstar}$), that $\Ustp$ and $\Vstp$ are balanced factors of (the nonzero block of) $\Mp$ in the sense that
\[
	\Ustp\transpose \Ustp = \Vstp\transpose \Vstp = \diag(d).
\]
The generalized Eckart-Young--type result \cite[Thm.~4.1]{Zhang2025c} (we omit the straightforward extension to the asymmetric case) then gives
\begin{align*}
	&\negqquad \normF{\Usmall \Vsmall\transpose - \Ustt \Vstt\transpose}^2 + \ip{\Ustp\transpose \Ustp}{\Vstt\transpose \Vstt} + \ip{\Vstp\transpose \Vstp}{\Ustt\transpose \Ustt} \\
	&= \normF{\Usmall \Vsmall\transpose - \Ustt \Vstt\transpose}^2 + \ip{\diag(d)}{\Ustt\transpose \Ustt + \Vstt\transpose \Vstt} \\
	&\geq \norm{s}^2 - \norm{(s' - d)_+}^2.
\end{align*}
Hence
\[
	\normF{\errmat}^2 \geq \norm{s}^2 - \norm{(s' - d)_+}^2 + \norm{d}^2.
\]
The rest of the proof of \Cref{lem:richard}, which consists purely of manipulations of these vector quantities,
follows exactly as in \cite[Sec.~4]{Zhang2025c}.

\subsection{Calculation of \texorpdfstring{$\mueff$}{μ\_eff}}
\label{sec:mueff_proof}
In this section, we prove \Cref{lem:opt_messy}.
To lighten notation, we set
\begin{equation*}
	\kapinv = \frac{\mu}{L} \quad \text{and} \quad \Lbr = \frac{ L_2}{L}.
\end{equation*}
Multiplying by $1/L$, the optimization problem becomes
\begin{equation}
	\label{eq:messyopt_simp}
	\min_{\substack{\alpha, \beta \geq 0 \\ \alpha^2 + \rho^2 \beta^2 \leq 1 + (\beta - \alpha)_+^2 }}~
	\max_{ \substack{t_1 \geq t_2 \geq 0 \\ (1 - \alpha^2) t_1^2 + \alpha^2 t_2^2 = 1} }~
	\frac{1 + \kapinv}{2} [(1 - \alpha^2) t_1 + \alpha^2 t_2] - \frac{1 - \kapinv}{2} - t_2 \Lbr \alpha \beta.
\end{equation}
We set
\[
	F(\alpha, \beta) \coloneqq \max_{ \substack{t_1 \geq t_2 \geq 0 \\ (1 - \alpha^2) t_1^2 + \alpha^2 t_2^2 = 1} }~\frac{1 + \kapinv}{2} [(1 - \alpha^2) t_1 + \alpha^2 t_2] - t_2 \Lbr \alpha \beta
\]
as the portion of the outer objective depending on $\alpha$ and $\beta$,
and we denote
\begin{equation}
	\label{eq:Fstar}
	\begin{aligned}
		F_* &\coloneqq \frac{\sqrt{(1 + \kapinv)^2 + \rho^{-2} \Lbr^2 } - \rho^{-1} \Lbr }{2} \\
		&= \sqrt{ \parens*{ \frac{1 + \kapinv}{2} }^2 + \frac{\rho^{-2} \Lbr^2}{2} - \frac{1 + \kapinv}{2} \rho^{-1} \Lbr \sqrt{1 + \frac{\rho^{-2} \Lbr^2}{(1 + \kapinv)^2}} }.
	\end{aligned}
\end{equation}

Our main task will be to show that
\begin{equation}
	\label{eq:Fab_optval}
	\min_{\substack{\alpha, \beta \geq 0 \\ \alpha^2 + \rho^2 \beta^2 \leq 1 + (\beta - \alpha)_+^2 }}~F(\alpha, \beta) = F_*.
\end{equation}
Along the way, we will show that this can always be done with
\begin{equation}
	\label{eq:t1_id}
	t_1 = \frac{1 + \kapinv}{2 F(\alpha, \beta)}
	\leq \frac{1+\kapinv}{2 F_*}.
\end{equation}

We first take care of several trivial cases:
\begin{itemize}
	\item $\Lbr = 0$,
	\item $\alpha = 0$, and
	\item $\alpha = 1$, in which case feasibility of $(\alpha, \beta)$ requires $\beta = 0$.
\end{itemize}
In each case, an optimal choice is $t_1 = t_2 = 1$, and we obtain
\[
	F(\alpha, \beta) = \frac{1 + \kapinv}{2} \geq F_*,
\]
and indeed \eqref{eq:t1_id} holds
(in fact, both are equalities if $\Lbr = 0$).
Thus, from now on, we will consider $\Lbr > 0$ and $0 < \alpha < 1$.

Under a change of variables $t_1 = \frac{\tau_1}{\sqrt{1 - \alpha^2}}$, $t_2 = \frac{\tau_2}{\alpha}$,
we have
\begin{equation*}
	F(\alpha, \beta) = \max_{ \tau_1^2 + \tau_2^2 = 1}~
	\frac{1 + \kapinv}{2} [\sqrt{1 - \alpha^2} \tau_1 + \alpha \tau_2] - \tau_2 \Lbr \beta
	\stquad \frac{\tau_1}{\sqrt{1 - \alpha^2}} \geq \frac{\tau_2}{\alpha} \geq 0.
\end{equation*}
Note that, for all $\alpha \in (0, 1)$,
\begin{equation}
	\label{eq:F_alpha_lb}
	\min_{\beta \geq 0}~F(\alpha, \beta) = \frac{1 + \kapinv}{2} \sqrt{1 - \alpha^2},
\end{equation}
and this is achieved for all $\beta$ such that $\Lbr \beta \geq \frac{1 + \kapinv}{2} \alpha$
with the choice $\tau_1 = 1$, $\tau_2 = 0$.
In this case, \eqref{eq:t1_id} again holds, as
\begin{align*}
	t_1 &= \frac{1}{\sqrt{1 - \alpha^2}} \\
	&= \frac{1 + \kapinv}{2 F(\alpha, \beta)}.
\end{align*}

In now only remains to handle the case $\Lbr \beta \leq \frac{1 + \kapinv}{2} \alpha$.
In this case,
\begin{equation}
	\label{eq:Fab_formula}
	\begin{aligned}
		F^2(\alpha, \beta) &= \parens*{ \frac{1 + \kapinv}{2} \sqrt{1 - \alpha^2} }^2 + \parens*{ \frac{1 + \kapinv}{2} \alpha - \Lbr \beta}^2 \\
		&= \parens*{ \frac{1 + \kapinv}{2} }^2 - (1 + \kapinv) \Lbr \alpha \beta + \Lbr^2 \beta^2,
	\end{aligned}
\end{equation}
and this holds for $\tau_1, \tau_2$ such that
\[
	\frac{\tau_2}{\tau_1}
	= \frac{\alpha - \frac{2 \Lbr}{1 + \kapinv} \beta}{\sqrt{1 - \alpha^2}}
	\quad \Longleftrightarrow \quad \frac{t_2}{t_1} = \frac{\sqrt{1 - \alpha^2}}{\alpha} \frac{\tau_2}{\tau_1} = 1 - \frac{2 \Lbr}{1 + \kapinv} \cdot \frac{\beta}{\alpha}.
\]
This implies
\begin{equation*}
	\begin{aligned}
		t_1^2
		&= \frac{t_1^2}{(1 - \alpha^2) t_1^2 + \alpha^2 t_2^2} \\
		&= \frac{1}{1 - \alpha^2 + \alpha^2 \parens*{ 1 - \frac{2 \Lbr}{1 + \kapinv} \cdot \frac{\beta}{\alpha} }^2} \\
		&= \frac{1}{1 - 2 \frac{2 \Lbr}{1 + \kapinv} \alpha \beta + \parens*{ \frac{2 \Lbr}{1 + \kapinv} }^2 \beta^2 } \\
		&= \parens*{ \frac{1 + \kapinv}{2} }^2 \cdot \frac{1}{F^2(\alpha, \beta)},
	\end{aligned}
\end{equation*}
so, once again,
\[
	t_1 = \frac{1 + \kapinv}{2 F(\alpha, \beta)}.
\]
This establishes \eqref{eq:t1_id}.

To calculate a (tight) lower bound on $F(\alpha, \beta)$ in the case $\beta \leq \frac{1 + \kapinv}{2 \Lbr} \alpha$,
we consider several subcases.
\begin{itemize}
	\item \textbf{Case 1:} $\alpha \leq \beta \leq \frac{1 + \kapinv}{2 \Lbr} \alpha$.
	This requires $\frac{1 + \kapinv}{2 \Lbr} \geq 1$.
	Feasibility of $(\alpha, \beta)$ implies
	\[
	1 \geq (\rho^2 - 1) \beta^2 + 2 \alpha \beta \geq (\rho^2 + 1)\alpha \beta.
	\]
	Then, from \eqref{eq:Fab_formula},
	\begin{align*}
		F^2(\alpha, \beta)
		&\geq \parens*{ \frac{1 + \kapinv}{2} }^2 - [(1 + \kapinv) \Lbr - \Lbr^2] \alpha \beta \\
		&\geq \parens*{ \frac{1 + \kapinv}{2} }^2 - \frac{(1 + \kapinv) \Lbr - \Lbr^2}{\rho^2 + 1} \\
		&\geq \parens*{ \frac{1 + \kapinv}{2} }^2 - \frac{1 + \kapinv}{2} \rho^{-1} \Lbr + \frac{\rho^{-2} \Lbr^2}{2} \\
		&\geq F_*^2.
	\end{align*}
	The second inequality used the fact that (in this case, as noted above) $\Lbr \leq 1 + \kapinv$.
	\item \textbf{Case 2:} $\beta \leq \alpha$, $\rho^{-1} \sqrt{1 - \alpha^2} \geq \frac{1 + \kapinv}{2 \Lbr} \alpha$.
	Then
	\[
	\alpha^2 \leq \frac{1}{1 + \parens*{ \frac{1 + \kapinv}{2 \rho^{-1} \Lbr} }^2 }
	\quad \Longrightarrow \quad 1 - \alpha^2 \geq \frac{1}{1 + \parens*{ \frac{2 \rho^{-1} \Lbr}{1 + \kapinv} }^2 },
	\]
	and then \eqref{eq:F_alpha_lb} implies
	\begin{align*}
		F(\alpha, \beta) &\geq \frac{1 + \kapinv}{2} \cdot \frac{1}{\sqrt{1 + \parens*{ \frac{2 \rho^{-1} \Lbr}{1 + \kapinv} }^2}} \\
		&\geq \frac{1 + \kapinv}{2} \cdot \frac{1}{\sqrt{1 + \frac{\rho^{-2} \Lbr^2}{(1 + \kapinv)^2} } + \frac{\rho^{-1} \Lbr}{1 + \kapinv} } \\
		&= F_*,
	\end{align*}
	where the last inequality uses the fact that $\sqrt{1 + 4 x^2} \leq \sqrt{1 + x^2} + x$ for $x \geq 0$.

	\item \textbf{Case 3:} $\beta \leq \alpha$, $\rho^{-1} \sqrt{1 - \alpha^2} \leq \frac{1 + \kapinv}{2 \Lbr} \alpha$.
	Feasibility then implies that
	\[
		\beta \leq \rho^{-1} \sqrt{1 - \alpha^2} \leq \frac{1 + \kapinv}{2 \Lbr} \alpha,
	\]
	so, noting that $F(\alpha, \beta)$ is monotonically decreasing in $\beta$, we have, from \eqref{eq:Fab_formula},
	\begin{align*}
		F^2(\alpha, \beta)
		&\geq \parens*{ \frac{1 + \kapinv}{2} }^2 - 2 \frac{1 + \kapinv}{2} \rho^{-1} \Lbr \alpha \sqrt{1 - \alpha^2} + \rho^{-2} \Lbr^2 (1 - \alpha^2) \\
		&= \parens*{ \frac{1 + \kapinv}{2} }^2 + \rho^{-2} \Lbr^2 (1 - \alpha^2) \\
		&\qquad + \frac{1 + \kapinv}{2} \rho^{-1} \Lbr \brackets*{ \parens*{ \sqrt{c} \sqrt{1 - \alpha^2} - \frac{\alpha}{\sqrt{c}} }^2 - \parens*{ c (1 - \alpha^2) + \frac{\alpha^2}{c}} }
	\end{align*}
	for any $c > 0$.
	To find a (tight) lower bound on this, we want to make $\parens*{ \sqrt{c} \sqrt{1 - \alpha^2} - \frac{\alpha}{\sqrt{c}} }^2$ be the only term depending on $\alpha$.
	This requires
	\[
	c - \frac{1}{c} = \frac{2 \rho^{-1} \Lbr}{1 + \kapinv}
	\qquad \Longleftrightarrow \qquad
	c = \frac{\rho^{-1} \Lbr}{1 + \kapinv} + \sqrt{1 + \frac{\rho^{-2} \Lbr^2}{(1 + \kapinv)^2}}.
	\]
	We then obtain
	\begin{align*}
		F^2(\alpha, \beta)
		&\geq \parens*{ \frac{1 + \kapinv}{2} }^2 + \rho^{-2} \Lbr^2 - \frac{1 + \kapinv}{2} \rho^{-1} \Lbr c \\
		&= \parens*{ \frac{1 + \kapinv}{2} }^2 + \frac{\rho^{-2} \Lbr^2}{2} - \frac{1 + \kapinv}{2} \rho^{-1} \Lbr \sqrt{1 + \frac{\rho^{-2} \Lbr^2}{(1 + \kapinv)^2}} \\
		&= F_*^2.
	\end{align*}
\end{itemize}
To show that the lower bound $F_*$ is indeed achieved,
take $\alpha_* \in (0, 1)$ such that
$\frac{\alpha_*}{\sqrt{1 - \alpha_*^2}} = c$ and $\beta_* = \rho^{-1} \sqrt{1 - \alpha_*^2}$.
Then
\[
\frac{\beta_*}{\alpha_*} = \frac{1}{\rho c} \leq \max\braces*{ 1, \frac{1 + \kapinv}{2 \Lbr} },
\]
and thus $(\alpha_*, \beta_*)$ is feasible, and $F(\alpha_*, \beta_*) = F_*$.

Finally, note that $F_* = \frac{1 - \kapinv}{2}$ (and hence the optimal value of the original problem is zero) if and only if
\begin{align*}
	(1 + \kapinv)^2 + \rho^{-2} \Lbr^2 &= [ \rho^{-1} \Lbr + (1 - \kapinv) ]^2 \\
	&= \rho^{-2} \Lbr^2 + (1 - \kapinv)^2 + 2 \rho^{-1} \Lbr (1 - \kapinv).
\end{align*}
Some algebra gives
\begin{gather*}
	4 \kapinv = 2 \rho^{-1} \Lbr(1 - \kapinv) \qquad \Longleftrightarrow \qquad 
	\kapinv = \frac{\Lbr}{2 \rho + \Lbr}.
\end{gather*}
Substituting $\kapinv = \mu/L$ and $\Lbr = L_2/L$
into \eqref{eq:t1_id} gives
\begin{gather*}
	t_1
	\leq \frac{1}{\sqrt{1 + \parens*{ \frac{\rho^{-1} \Lbr}{1 + \kapinv} }^2 } - \frac{\rho^{-1} \Lbr}{1 + \kapinv} }
	=  \sqrt{1 + \parens*{ \frac{\rho^{-1} L_2}{L + \mu} }^2 } + \frac{\rho^{-1} L_2}{L + \mu},
\end{gather*}
and multiplying \eqref{eq:messyopt_simp} by $L$ and substituting in \eqref{eq:Fstar} and \eqref{eq:Fab_optval} gives an optimum value of
\begin{align*}
	L \parens*{ F_* - \frac{1 - \kapinv}{2} }
	&= L \parens*{ \frac{\sqrt{(1 + \kapinv)^2 + \rho^{-2} \Lbr^2 } - \rho^{-1} \Lbr }{2} - \frac{1 - \kapinv}{2} } \\
	&= \frac{\sqrt{(L + \mu)^2 + \rho^{-2} L_2^2} - \rho^{-1} L_2 - (L - \mu)}{2}.
\end{align*} 
This completes the proof of \Cref{lem:opt_messy}.

\section{Counterexample construction}
\label{sec:proof_counter}
\subsection{General counterexample}

The following lemma gives a general class of counterexamples that we can then specialize for various purposes:
\begin{lemma}
	\label{lem:spur_gen}
	In the symmetric case (respectively, in the asymmetric case),
	for any
	\begin{itemize}
		\item Integers $\rstar \geq 1$, $\rmax \geq \rstar$, and $d \geq \rstar + \rmax$,
		\item Reals $\epsilon \in (0, 1)$ and $c \geq \cp > 0$ satisfying
		\[
			c^2 \rstar + \cp^2 \rmax = 1 - \epsilon,
		\]
	\end{itemize}
	there is $\Mst \in \Mspaced$ (resp.\ $\symms_d^+$) with rank $\rstar$ and $\sigma_1(\Mst) = \sigma_{\rstar}(\Mst) = 1$,
	an integer $n$, and a linear operator $\scrA \colon \Mspaced \to \R^n$ (resp.\ $\scrA \colon \symms_d \to \R^n$) such that the following hold:
	\begin{itemize}
		\item $L_{\genrk}(\scrA) = 1$ and $\mu_{\genrk}(\scrA) \geq \epsilon$ for all $\genrk \geq 1$, and, for all integers $r \in [\rstar, \rmax]$, we have
		\begin{equation*}
			\mu_{r + \rstar}(\scrA) = 1 - c^2 \rstar - \cp^2 r.
		\end{equation*}
		
		\commonpts
	\end{itemize}
	If, in addition, for some integer $r \in [\rstar, \rmax]$,
	\begin{equation}
		\label{eq:spur_cond}
		c \cp \rstar \geq 1 - c^2 \rstar - \cp^2 r
		= \epsilon  + \cp^2(\rmax - r),
	\end{equation}
	then
	\begin{itemize}
		\item (Asymmetric case) The nonconvex problem \eqref{eq:ncvx_asym} with the given $\phi$, $\lambda$, and $r$ has a second-order critical point $(U, V)$ with $U \neq 0$, $V \neq 0$ such that $U\transpose \Mst = 0$ and $\Mst V = 0$, or, respectively,
		
		\item (Symmetric case) The nonconvex problem \eqref{eq:ncvx_sym} with the given $\phi$, $\lambda$, and $r$ has a second-order critical point $U \neq 0$ with $\Mst U = 0$.
		
		\item Furthermore, if the inequality \eqref{eq:spur_cond} is strict,
		then $(U, V)$ (resp.\ $U$) is a local minimum.
	\end{itemize}
\end{lemma}
We prove this in \Cref{sec:proof_counter_lemma} below.
The counterexample construction is a generalization of those appearing in \cite{Zhang2025c,Zhang2025a}.

\subsection{Proofs of concrete counterexamples}
In this section, we use \Cref{lem:spur_gen} to prove the two counterexample results given in \Cref{sec:res_counter}.

\begin{proof}[Proof of \Cref{thm:optmucond}]
In \Cref{lem:spur_gen},	we take $\rmax = r$ and $\epsilon = \mu > 0$.
Note that this ensures $\mu_{\genrk}(\scrA) \geq \mu$ for all $\genrk \geq 1$.
The constraint on $c \geq \cp \geq 0$ is then
\[
	c^2 \rstar + \cp^2 r = 1 - \mu.
\]
We then calculate
\begin{align*}
	&\negqquad \max_{c, \cp}~c \cp \rstar \stquad c \geq \cp \geq 0,\ c^2 \rstar + \cp^2 r = 1 - \mu \\
	&= \max_{c, \cp}~( c \sqrt{\rstar} )(\cp \sqrt{r}) \sqrt{\frac{\rstar}{r}} \stquad c \geq \cp \geq 0,\ \frac{c^2 \rstar + \cp^2 r}{2} = \frac{1 - \mu}{2} \\
	&= \frac{1 - \mu}{2} \sqrt{\frac{\rstar}{r}}
\end{align*}
by the arithmetic-geometric mean inequality,
choosing $c, \cp$ such that $c \sqrt{\rstar} = \cp \sqrt{r}$ (we indeed have $c \geq \cp$ because $r \geq \rstar$).
With these choices, the condition \eqref{eq:spur_cond} becomes
\[
	\frac{1 - \mu}{2} \sqrt{\frac{\rstar}{r}}
	\geq \mu,
\]
which is equivalent to \eqref{eq:mu_counter} (including when the inequalities are strict).
The result follows by \Cref{lem:spur_gen}.
\end{proof}

\begin{proof}[Proof of \Cref{thm:overpbad}]
We now choose, in \Cref{lem:spur_gen}, $r = \rmax = \rbig$,
some $\epsilon \in (0, \mu)$ that we will further specify, and $c, \cp > 0$ (we will later show that $c \geq \cp$ under our assumptions) satisfying
\begin{align*}
	c^2 \rstar + \cp^2 \rsmall &= 1 - \mu, \quad \text{and} \\
	c^2 \rstar + \cp^2 \rbig &= 1 - \epsilon.
\end{align*}
Hence the resulting $\scrA$ will satisfy $\mu_{\rsmall+\rstar}(\scrA) = \mu$.
We can partially solve these equations for $c^2$ and $\cp^2$ to obtain
\begin{align*}
	\cp^2 (\rbig - \rsmall) &= \mu - \epsilon, \quad \text{and} \\
	c^2 \rstar \parens*{\frac{\rbig}{\rsmall} - 1} &= \frac{\rbig}{\rsmall}(1 - \mu) - (1 - \epsilon).
\end{align*}
Note that $\cp > 0$, and, if $c > 0$, then, for sufficiently small $\epsilon > 0$, we will have $c \cp \rstar > \epsilon$, satisfying the condition \eqref{eq:spur_cond} with strict inequality.
Thus it suffices to show $c \geq \cp$.
Noting from the previous equations that
\[
	\frac{\cp^2}{c^2} = \frac{\frac{\mu - \epsilon}{\rbig - \rsmall}}{\frac{\rsmall}{\rstar (\rbig - \rsmall)} \cdot \parens*{\frac{\rbig}{\rsmall}(1 - \mu) - (1 - \epsilon)}}
	= \rstar \cdot \frac{\mu - \epsilon}{\rbig(1 - \mu) - \rsmall (1 - \epsilon)},
\]
we will have $c \geq \cp$ for any $\epsilon > 0$
if
\begin{gather*}
	\mu \leq \frac{\rbig(1 - \mu) - \rsmall}{\rstar}
	\quad \Longleftrightarrow \quad
	\rbig \geq \frac{\rstar \mu + \rsmall}{1 - \mu}.
\end{gather*}
This completes the proof.
\end{proof}

\subsection{Proof of general counterexample}
\label{sec:proof_counter_lemma}
In this section, we provide a proof of \Cref{lem:spur_gen}.
The counterexample construction generalizes that of \cite{Zhang2025c,Zhang2025a}.

\subsubsection{Symmetric case}
We first prove \Cref{lem:spur_gen} in the symmetric case.
The asymmetric case will be a slight adaptation.

Let $Q \in \R^{d \times \rstar}$, $\Qp \in \R^{d \times \rmax}$ such that $[Q\ \Qp] \in \R^{d \times (\rstar + \rmax)}$ has orthonormal columns.
Denote $\PQ \coloneqq Q Q\transpose$ and $\PQp \coloneqq \Qp \Qp\transpose$;
these are the orthogonal projection matrices onto $\range(Q)$ and $\range(\Qp)$ respectively.
Set $G = c \PQ - \cp \PQp$.
As
\[
	\normF{G}^2 = c^2 \rstar + \cp^2 \rmax
	= 1 - \epsilon < 1,
\]
the quadratic form $\genmat \mapsto \normF{\genmat}^2 - \ip{G}{\genmat}^2$ on $\symms_d$ is (strictly) positive definite, so,
for some integer $n$, there is a map $\scrA \colon \symms_d \to \R^n$ such that, for all $\genmat \in \symms_d$,
\[
	\norm{\scrA(\genmat)}^2 = \normF{\genmat}^2 - \ip{G}{\genmat}^2.
\]
We must furthermore have
\[
\scrA^* \scrA(\genmat) = \genmat - \ip{G}{\genmat} G.
\]
For all $\genrk \geq 1$, $L_{\genrk}(\scrA) = 1$, and, by the bound on $\normF{G}^2$, $\mu_{\genrk}(\scrA) \geq \epsilon$.
Furthermore, for $\genrk \leq \rstar + \rmax$ (recalling that we assumed $c \geq \cp$),
\begin{align*}
\mu_{\genrk}(\scrA)
&= 1 - \sum_{\ell = 1}^\genrk \sigma_\ell^2(G) \\
&= 1 - \parens*{ c^2 \min\{ \genrk, \rstar \} + \cp^2 \max\{ \genrk - \rstar, 0 \} }.
\end{align*}
In particular, as $\rstar \leq r \leq \rmax$,
\[
\mu_{r + \rstar}(\scrA) = 1 - c^2 \rstar - \cp^2 r
\]
as claimed.

We will set $M_* = Q Q\transpose = \PQ$.
For $\lambda \geq 0$,
we construct $b \in \R^n$ such that (recall \eqref{eq:quadphi_grad})
\begin{equation}
	\label{eq:y_dualcert}
	\nabla \phi(\Mst) + \lambda I_d = \scrA^* (\scrA(\Mst) - b) + \lambda I_d = \lambda (I_d - \PQ - \PQp).
\end{equation}
Recalling the optimiality conditions \eqref{eq:global_sym_comp} and \eqref{eq:global_sym_dualfeas},
if \eqref{eq:y_dualcert} holds, then $\Mst$ is the global optimum (unique due to the strong convexity of $\phi$) of \eqref{eq:cvx_sym}.
Furthermore, \eqref{eq:y_dualcert} is equivalent to $\scrA^*(\xi) = \scrA^*(b - \scrA(\Mst)) = \lambda(\PQ + \PQp)$;
note that this implies $\opnorm{\scrA^*(\xi)} = \lambda$ as claimed.

To show that we can make \eqref{eq:y_dualcert} hold,
we choose $b = \scrA((1 + a) \PQ + \ap \PQp)$ for some $a, \ap \in \R$.
Then
\begin{align*}
\scrA^* (b - \scrA(\Mst))
&= \scrA^* \scrA( a \PQ + \ap \PQp ) \\
&= a \PQ + \ap \PQp - \ip{G}{a \PQ + \ap \PQp} G \\
&= a \PQ + \ap \PQp - (c \rstar a - \cp \rmax \ap)(c \PQ - \cp \PQp) \\
&= [(1 - c^2 \rstar)a + c \cp \rmax \ap] \PQ + [ (1 - \cp^2 \rmax) \ap + c \cp \rstar a ] \PQp.
\end{align*}
It then suffices to solve the linear system
\[
	\begin{bmatrix*}
		1 - c^2 \rstar & c \cp \rmax \\
		c \cp \rstar & 1 - \cp^2 \rmax
	\end{bmatrix*}
	\begin{bmatrix*}
		a \\ \ap
	\end{bmatrix*}
	=
	\begin{bmatrix*}
		\lambda \\ \lambda
	\end{bmatrix*}
\]
for $a, \ap$,
and then $\scrA^* (b - \scrA(\Mst)) = \lambda(\PQ + \PQp)$, and \eqref{eq:y_dualcert} will be satisfied.
This is possible if the matrix is invertible; indeed,
\begin{align*}
\det\parens*{\begin{bmatrix*}
		1 - c^2 \rstar & c \cp \rmax \\
		c \cp \rstar & 1 - \cp^2 \rmax
\end{bmatrix*} }
&= (1 - c^2 \rstar)(1 - \cp^2 \rmax) - c^2 \cp^2 \rmax \rstar \\
&= 1 - c^2 \rstar - \cp^2 \rmax \\
&= \epsilon \\
&> 0.
\end{align*}
Thus we have established \eqref{eq:y_dualcert}.
Note, furthermore, that if $\lambda = 0$,
then $a = \ap = 0$,
so $b = \scrA(\Mst)$.

It now remains to show that, if \eqref{eq:spur_cond} holds,
the nonconvex problem \eqref{eq:ncvx_sym} has a spurious second-order critical point (local minimum if \eqref{eq:spur_cond} is strict).
We will consider points of the form $U = \Qp R$ for some $R \in \R^{\rmax \times r}$.
We calculate, by \eqref{eq:quadphi_grad} and \eqref{eq:y_dualcert},
\begin{align*}
\nabla \phi(U U\transpose) + \lambda I_d
&= (\nabla \phi(U U\transpose) - \nabla \phi(M_*) )+ ( \nabla \phi(M_*) + \lambda I_d) \\
&= \scrA^*\scrA(U U\transpose - \PQ) + \lambda (I_d - \PQ - \PQp) \\
&= U U\transpose - \PQ - \ip{G}{U U\transpose - \PQ} G + \lambda (I_d - \PQ - \PQp) \\
&= \Qp R R\transpose \Qp\transpose - \PQ - (-\cp \normF{R}^2 - c \rstar)(c \PQ - \cp \PQp) \\
&\qquad+ \lambda (I_d - \PQ - \PQp) \\
&= -(1 - c^2 \rstar - c \cp \normF{R}^2) \PQ +\Qp[ R R\transpose - (\cp^2 \normF{R}^2 + c \cp \rstar)I_{\rmax} ] \Qp\transpose \\
&\qquad+ \lambda (I_d - \PQ - \PQp).
\end{align*}
For $\nabla f_\lambda(U)$ (given by \eqref{eq:grad_sym}) to be zero, we need
\begin{align*}
0 &= (\nabla \phi(U U\transpose) + \lambda I_d) \Qp R \\
&= \Qp [ R R\transpose R - (\cp^2 \normF{R}^2 + c \cp \rstar) R ].
\end{align*}
As $r \leq \rmax$, a natural choice is $R \in \R^{\rmax \times r}$ such that $R\transpose R = x I_r$ for some $x \geq 0$.
This implies $\normF{R}^2 = r x$.
With this choice, the previous condition becomes
\[
0 = x - \cp^2 r x - c \cp \rstar,
\]
which we can solve for
\begin{equation}
\label{eq:RTR}
x = \frac{c \cp \rstar}{1 - \cp^2 r} \qquad \Longrightarrow \qquad R\transpose R = \frac{c \cp \rstar}{1 - \cp^2 r} I_r.
\end{equation}
With this choice,
\begin{align*}
	R R\transpose - (\cp^2 \normF{R}^2 + c \cp \rstar)I_{\rmax}
	&= - x (I_{\rmax} - R R^\dagger) \\
	&= - \frac{c \cp \rstar}{1 - \cp^2 r} (I_{\rmax} - R R^\dagger),
\end{align*}
and
\begin{align*}
	1 - c^2 \rstar - c \cp \normF{R}^2
	&= 1 - c^2 \rstar - c \cp r \frac{c \cp \rstar}{1 - \cp^2 r} \\
	&= \frac{1 - c^2 \rstar - \cp^2 r}{1 - \cp^2 r}.
\end{align*}

Therefore,
\begin{equation}
\label{eq:spur_grad_calc}
\begin{aligned}
	\nabla \phi(U U\transpose) + \lambda I_d
	&= - \frac{1 - c^2 \rstar - \cp^2 r}{1 - \cp^2 r} \PQ - \frac{c \cp \rstar}{1 - \cp^2 r} \Qp (I_{\rmax} - R R^\dagger)\Qp\transpose \\
	&\qquad + \lambda (I_d - \PQ - \PQp).
\end{aligned}
\end{equation}
The Hessian quadratic form applied to $\Udt \in \Uspace$ is then (recall \eqref{eq:hess_sym} and \eqref{eq:quadphi_hess})
\begin{align*}
\nabla^2 f_\lambda(U)[\Udt, \Udt]
&= 2 \ip{\nabla \phi(U U\transpose) + \lambda I_d}{\Udt \Udt\transpose} + \norm{\scrA(U \Udt\transpose + \Udt U\transpose)}^2 \\
&= - 2 \frac{1 - c^2 \rstar - \cp^2 r}{1 - \cp^2 r}\ip{\PQ}{\Udt \Udt\transpose} - 2 \frac{c \cp \rstar}{1 - \cp^2 r} \ip{\Qp (I_{\rmax} - R R^\dagger)\Qp\transpose}{\Udt \Udt\transpose} \\
&\qquad + 2 \lambda \ip{I_d - \PQ - \PQp}{\Udt \Udt\transpose} + \norm{\scrA(\Qp R \Udt\transpose + \Udt R\transpose \Qp\transpose)}^2.
\end{align*}
We can write any $\Udt \in \Uspace$ as
\[
\Udt = Q \Sdt + \Qp \Rdt + \Wdt,
\]
where $\Sdt \in \R^{\rstar \times r}$, $\Rdt \in \R^{\rmax \times r}$, and $\Wdt \in \range(I_d - \PQ - \PQp)$.
We further decompose
\[
\Rdt = \Rdt_R + \Rdt_\perp,
\]
where $\range(\Rdt_R) \subseteq \range(R)$ and $\Rdt_\perp\transpose R = 0$.
Note that we then have
\begin{align*}
\ip{\PQ}{\Udt \Udt\transpose} &= \normF{\Sdt}^2, \\
\ip{I_d - \PQ - \PQp}{\Udt \Udt\transpose} &= \normF{\Wdt}^2, \quad \text{and} \\
\ip{\Qp (I_{\rmax} - R R^\dagger)\Qp\transpose}{\Udt \Udt\transpose} &= \normF{\Rdt_\perp}^2.
\end{align*}
Furthermore,
\begin{align*}
\norm{\scrA(\Qp R \Udt\transpose + \Udt R\transpose \Qp\transpose)}^2
&= \normF{ \Qp R \Udt\transpose + \Udt R\transpose \Qp\transpose }^2 - \ip{G}{\Qp R \Udt\transpose + \Udt R\transpose \Qp\transpose}^2 \\
&= 2 (\normF{\Sdt R\transpose}^2 + \normF{\Wdt R\transpose}^2 + \normF{\Rdt_\perp R\transpose}^2 ) \\
&\qquad + \normF{R \Rdt_R\transpose + \Rdt_R R\transpose}^2 - \cp^2 \tr^2( R \Rdt_R\transpose + \Rdt_R R\transpose ) \\
&\geq 2 \frac{c \cp \rstar}{1 - \cp^2 r}(\normF{\Sdt}^2 + \normF{\Wdt}^2 + \normF{\Rdt_\perp}^2) + (1 - \cp^2 r) \normF{R \Rdt_R\transpose + \Rdt_R R\transpose}^2.
\end{align*}
The inequality uses \eqref{eq:RTR} and the fact that $\tr^2( B ) \leq r \normF{B}^2$ for any $B \in \R^{r \times r}$.
Plugging these into the previous expression for $\nabla^2 f_\lambda(U)$,
we have
\begin{equation}
	\label{eq:counter_hess_final}
	\begin{aligned}
		\nabla^2 f_\lambda(U)[\Udt, \Udt]
		&\geq  2 \parens*{ \frac{c \cp \rstar}{1 - \cp^2 r} - \frac{1 - c^2 \rstar - \cp^2 r}{1 - \cp^2 r} } \normF{\Sdt}^2 \\
		&\qquad + 2 \parens*{ \frac{c \cp \rstar}{1 - \cp^2 r} - \frac{c \cp \rstar}{1 - \cp^2 r} } \normF{\Rdt_\perp}^2 \\
		&\qquad + 2\parens*{ \lambda + \frac{c \cp \rstar}{1 - \cp^2 r}}  \normF{\Wdt}^2 + (1 - \cp^2 r) \normF{R \Rdt_R\transpose + \Rdt_R R\transpose}^2 \\
		&= 2 \frac{ c \cp \rstar - (1 - c^2 \rstar - \cp^2 r)}{1 - \cp^2 r} \normF{\Sdt}^2 \\
		&\qquad + 2\parens*{ \lambda + \frac{c \cp \rstar}{1 - \cp^2 r}}  \normF{\Wdt}^2 + (1 - \cp^2 r) \normF{R \Rdt_R\transpose + \Rdt_R R\transpose}^2.
	\end{aligned}
\end{equation}
$\nabla^2 f_\lambda(U)$ is thus positive semidefinite if and only if
\[
c \cp \rstar \geq 1 - c^2 \rstar - \cp^2 r,
\]
which is precisely \eqref{eq:spur_cond}.

To show that $U$ is a local minimum if \eqref{eq:spur_cond} is strict,
it suffices to check that $f_\lambda$ is nondecreasing in the direction of $\Udt$ for which $\nabla^2 f_\lambda(U)[\Udt, \Udt] = 0$.
For such $\Udt$,
note that $f_\lambda(U + \Udt) - f_\lambda(U)$ is a fourth-order polynomial in $\Udt$,
and, because $\nabla f_\lambda(U) = 0$ and $\nabla^2 f_\lambda(U)[\Udt, \Udt] = 0$,
all terms of order 2 or less in $\Udt$ are zero,
leaving only the third- and fourth-order terms.
One can then easily calculate
\begin{equation}
	\label{eq:inc_deg4}
	f_\lambda(U + \Udt) - f_\lambda(U)
	= \frac{1}{2} \norm{\scrA(\Udt \Udt\transpose)}^2 + \ip{\scrA(U \Udt\transpose + \Udt U\transpose)}{\scrA(\Udt \Udt\transpose)}
\end{equation}
We now show that this is positive.

As \eqref{eq:spur_cond} is strict, by \eqref{eq:counter_hess_final}, the only $\Udt$ such that $\nabla^2 f_\lambda(U)[\Udt, \Udt] = 0$ have the form $\Udt = \Qp (\Rdt_R + \Rdt_\perp)$,
where $\Rdt_R$ must satisfy $R \Rdt_R\transpose + \Rdt_R R\transpose = 0$.
In particular, for such $\Udt$,
\begin{equation}
	\label{eq:UdtUdt}
	\Udt \Udt\transpose = \Qp ( \Rdt_R \Rdt_R\transpose + \Rdt_R \Rdt_\perp\transpose + \Rdt_\perp \Rdt_R\transpose + \Rdt_\perp \Rdt_\perp\transpose ) \Qp\transpose,
\end{equation}
and
\begin{equation}
	\label{eq:UdtU}
	U \Udt\transpose + \Udt U\transpose = \Qp( R \Rdt_\perp\transpose + \Rdt_\perp R\transpose ) \Qp\transpose.
\end{equation}
Note, from this last expression, that
\[
	\ip{G}{U \Udt\transpose + \Udt U\transpose} = - \cp \tr(R \Rdt_\perp\transpose + \Rdt_\perp R\transpose) = - 2 \cp \tr(\Rdt_\perp\transpose R) = 0.
\]
Therefore
\begin{equation}
	\label{eq:Udt_Aeq}
	\scrA^*\scrA(U \Udt\transpose + \Udt U\transpose) = U \Udt\transpose + \Udt U\transpose.
\end{equation}

From \eqref{eq:UdtUdt} and \eqref{eq:UdtU},
we calculate
\begin{align*}
\ip{U \Udt\transpose + \Udt U\transpose}{\Udt \Udt\transpose}
&= \ip{ R \Rdt_\perp\transpose + \Rdt_\perp R\transpose }{ \Rdt_R \Rdt_R\transpose + \Rdt_R \Rdt_\perp\transpose + \Rdt_\perp \Rdt_R\transpose + \Rdt_\perp \Rdt_\perp\transpose } \\
&= \ip{R \Rdt_\perp\transpose}{\Rdt_R \Rdt_\perp\transpose} + \ip{\Rdt_\perp R\transpose}{\Rdt_\perp \Rdt_R\transpose} \\
&= \ip{\Rdt_\perp\transpose \Rdt_\perp}{R\transpose \Rdt_R + \Rdt_R\transpose R}.
\end{align*}
Note, recalling \eqref{eq:RTR}, that, for some $x > 0$,
\begin{align*}
R\transpose \Rdt_R + \Rdt_R\transpose R
&= \frac{1}{x} \parens{ R\transpose \Rdt_R R\transpose R + R\transpose R \Rdt_R\transpose R } \\
&= \frac{1}{x} R\transpose(\Rdt_R R\transpose + R \Rdt_R\transpose) R \\
&= 0.
\end{align*}
Therefore, additionally using \eqref{eq:Udt_Aeq}, we obtain
\begin{equation}
	\label{eq:thirdorder_zero}
	\ip{\scrA(U \Udt\transpose + \Udt U\transpose)}{\scrA(\Udt \Udt\transpose)} = \ip{U \Udt\transpose + \Udt U\transpose}{\Udt \Udt\transpose} = 0.
\end{equation}
Substituting this into \eqref{eq:inc_deg4}, we have
\[
	f_\lambda(U + \Udt) - f_\lambda(U)
	= \frac{1}{2} \norm{\scrA(\Udt \Udt\transpose)}^2 \geq 0.
\]
We thus conclude that $U$ is a local minimum of $f_\lambda$.
This completes the proof of \Cref{lem:spur_gen} in the symmetric case.

\subsubsection{Extension to the asymmetric case}
In the asymmetric case, we use a near-identical construction as in the symmetric case above.
We consider the same $Q$ and $\Qp$ and the same (symmetric) $G$ and $\Mst$.
A subtle change is that we now define $\scrA$ implicitly over $\R^{d \times d}$ (rather than $\symms_d$) by
\[
\normF{\scrA(\genmat)}^2 = \normF{\genmat}^2 - \ip{G}{\genmat}^2
\]
for any $\genmat \in \R^{d \times d}$.
Critically, note that
\begin{equation}
\label{eq:Aineq_sym}
\normF{\scrA(\genmat)}^2 \geq \norm*{\scrA\parens*{\frac{\genmat + \genmat\transpose}{2}}}^2,
\end{equation}
and this inequality is \emph{strict} unless $\genmat$ is symmetric.

We similarly set
\[
	b = \scrA((1 + a) \PQ + \ap \PQp)
\]
for the exact same $a, \ap$ as in the symmetric case.
Then, we again have \eqref{eq:y_dualcert}, which implies
\[
\nabla \phi(\Mst) = \lambda(\PQ + \PQp) \in \lambda \partial \nucnorm{\Mst}.
\]
Thus $\Mst$ is a global optimum of \eqref{eq:cvx_asym},
and it is unique due to the strong convexity of $\phi$.
We then show that, for the same $U$ as in the symmetric case, $(U, U)$ is a second-order critical point (local optimum if \eqref{eq:spur_cond} is strict) of \eqref{eq:ncvx_asym}.

Recalling \eqref{eq:grad_asym}, showing $\nabla f_\lambda(U, U) = 0$ is identical to the symmetric case.
It thus remains to consider the Hessian.
By \eqref{eq:spur_grad_calc} (which still holds),
we have
\begin{equation}
\label{eq:negdef_grad}
\nabla \phi(U U\transpose)
= - \frac{1 - c^2 \rstar - \cp^2 r}{1 - \cp^2 r} \PQ - \frac{c \cp \rstar}{1 - \cp^2 r} \Qp (I_{\rmax} - R R^\dagger)\Qp\transpose - \lambda (\PQ + \PQp) \preceq 0.
\end{equation}
By \eqref{eq:hess_asym} and \eqref{eq:quadphi_hess}, for $\Udt, \Vdt \in \Uspace$,
\begin{align*}
&\negqquad \nabla^2 f_\lambda(U, U)[(\Udt, \Vdt), (\Udt, \Vdt)] \\
&= 2 \parens*{ \ip{\nabla \phi(U U\transpose)}{\Udt \Vdt\transpose} + \frac{\lambda}{2} (\normF{\Udt}^2 + \normF{\Vdt}^2) } 
+ \norm{\scrA( U \Vdt\transpose + \Udt U\transpose )}^2.
\end{align*}
Set, for convenience,
\[
\Adt \coloneqq \frac{\Udt + \Vdt}{2}, \qquad \Ddt \coloneqq \frac{\Udt - \Vdt}{2}.
\]
Note that $\Udt \Vdt\transpose + \Vdt \Udt\transpose \preceq 2 \Adt \Adt\transpose$.
Then, by \eqref{eq:negdef_grad},
we have
\[
\ip{\nabla \phi(U U\transpose)}{\Udt \Vdt\transpose} = \frac{1}{2} \ip{\nabla \phi(U U\transpose)}{\Udt \Vdt\transpose + \Vdt \Udt\transpose} \geq \ip{ \nabla \phi(U U\transpose) }{\Adt \Adt\transpose}.
\]
Next, note that
\[
\normF{\Udt}^2 + \normF{\Vdt}^2 \geq \frac{1}{2} \normF{\Udt + \Vdt}^2 = 2 \normF{\Adt}^2.
\]
Furthermore, by \eqref{eq:Aineq_sym},
\begin{equation}
\label{eq:Aineq2}
\begin{aligned}
	&\negqquad \norm{\scrA( U \Vdt\transpose + \Udt U\transpose )}^2 \\
	&\geq \norm*{\scrA\parens*{ \frac{ U \Vdt\transpose + \Udt U\transpose + \Vdt U\transpose + U \Udt\transpose }{2}}}^2 \\
	&= \norm{ \scrA\parens{ U \Adt\transpose + \Adt U\transpose } }^2.
\end{aligned}
\end{equation}
We can thus conclude that
\[
\nabla^2 f_\lambda(U, U)[(\Udt, \Vdt), (\Udt, \Vdt)] \geq
\nabla^2 f_\lambda(U, U)[(\Adt, \Adt), (\Adt, \Adt)] \geq 0,
\]
where the second inequality is from the symmetric case.
Thus $(U, U)$ is a second-order critical point of \eqref{eq:ncvx_asym}.

Finally, we show that, if the inequality \eqref{eq:spur_cond} is strict, $(U, U)$ is in fact a local optimum.
Similarly to the symmetric case, it suffices to show that, for any $\Udt, \Vdt$ such that $\nabla^2 f_\lambda(U, U)[(\Udt, \Vdt), (\Udt, \Vdt)] = 0$,
we have $f_\lambda(U + \Udt, U + \Vdt) - f_\lambda(U, U) \geq 0$.

Note, by \eqref{eq:Aineq_sym} and the note after it,
that the inequality \eqref{eq:Aineq2} is strict unless $U \Vdt\transpose + \Udt U\transpose$ is symmetric.
Then $\nabla^2 f_\lambda(U, U)[(\Udt, \Vdt), (\Udt, \Vdt)] = 0$ implies
\[
	U \Vdt\transpose + \Udt U\transpose = \Vdt U\transpose + U \Udt\transpose \quad \Longrightarrow \quad U(\Udt - \Vdt) = (\Udt - \Vdt)U\transpose \quad \Longrightarrow \quad \Ddt = \frac{\Udt - \Vdt}{2} = U \Sdt
\]
for some \emph{symmetric} $\Sdt \in \R^{r \times r}$.

Furthermore, as we must have $\nabla^2 f_\lambda(U, U)[(\Adt, \Adt), (\Adt, \Adt)] = 0$,
the calculations from the symmetric case applied to $\Adt$ imply that (by \eqref{eq:thirdorder_zero})
\begin{equation*}
\ip{\scrA(U \Adt\transpose + \Adt U\transpose)}{\scrA(\Adt \Adt\transpose)} = 0
\end{equation*}
and (by \eqref{eq:UdtU} and \eqref{eq:Udt_Aeq})
\begin{equation*}
\begin{aligned}
	\ip{\scrA(U \Adt\transpose + \Adt U\transpose)}{\scrA(\Ddt \Ddt\transpose)}
	&= \ip{U \Adt\transpose + \Adt U\transpose}{U \Sdt^2 U\transpose} \\
	&= \ip{R \Rdt_\perp\transpose + \Rdt_\perp R\transpose}{R \Sdt^2 R\transpose} \\
	&= 0.
\end{aligned}
\end{equation*}
We then can directly calculate, once again dropping all terms of degree 2 or less in $\Udt$ and $\Vdt$,
\begin{align*}
&\negqquad f_\lambda(U + \Udt, U + \Vdt) - f_\lambda(U, U) \\
&= \frac{1}{2} \norm{\scrA(\Udt \Vdt\transpose)}^2 + \ip{\scrA(U \Vdt\transpose + \Udt U\transpose)}{\scrA(\Udt \Vdt\transpose)} \\
&\geq \ip{\scrA(U (\Adt - \Ddt)\transpose + (\Adt + \Ddt) U\transpose)}{\scrA((\Adt + \Ddt)(\Adt - \Ddt)\transpose)} \\
&= \ip{\scrA(U \Adt + \Adt U\transpose + \underbrace{\Ddt U\transpose - U \Ddt\transpose}_{=0})}{\scrA(\Adt \Adt\transpose + \underbrace{\Ddt \Adt\transpose - \Adt \Ddt\transpose}_{\text{skew}} - \Ddt \Ddt\transpose)} \\
&= \ip{\scrA(U \Adt + \Adt U\transpose)}{\scrA(\Adt \Adt\transpose)} - \ip{\scrA(U \Adt + \Adt U\transpose)}{\scrA(\Ddt \Ddt\transpose)} \\
&= 0.
\end{align*}
This completes the proof of \Cref{lem:spur_gen}.

\section*{Acknowledgments and declarations}
The authors thank Irène Waldspurger and Yann Traonmilin for helpful discussions and suggestions.
AM was supported by the Hi!~PARIS and ANR/France 2030 program (ANR-23-IACL-0005)
and the Swiss SERI (contract MB22.00027).
RZ was supported by NSF CAREER Award ECCS-2047462 and ONR Award N00014-24-1-2671.
The authors have no relevant financial or non-financial interests to disclose.

\bibliographystyle{ieeetr}
\bibliography{refs}

\end{document}